\documentclass[a4paper]{article}

\usepackage{amsmath,amssymb,graphics,epsfig,color,cite}
\usepackage{mathrsfs,bbm}
\usepackage{amsthm}

\usepackage{bm}

\usepackage{tikz}

\usepackage{hyperref}
 
\newcommand{\mbf}[1]{\mathbf{#1}}

\DeclareMathOperator{\range}{Ran}

\newcommand{\lp}{\langle}
\newcommand{\rp}{\rangle}

\newcommand{\ve}{\varepsilon}

\DeclareMathOperator*{\argmin}{arg\,min}

\def\Hess{{\rm Hess\,}}

\def\supp{\mathop{\rm supp} \nolimits} 
\def\Ran{{\rm Ran}}
\def\sspan{{\rm Span}}

\def\and {{\rm \; and \;}}

\def\dim {{\rm \; dim  \;}}

\newcommand {\pa}{\partial}

\newcommand{\ft}[1]{\mathsf{#1}} 

\newtheorem{theorem}{Theorem}
\newtheorem{proposition}{Proposition}

\newtheorem{lemma}{Lemma}

\newtheorem{remark}{Remark}

\usepackage[hang,font={normalsize,it},labelfont=sf]{caption}
\usepackage{enumitem}
 
 \usepackage{titlesec}

\titleformat*{\section}{\normalsize\bfseries}
\titleformat*{\subsection}{\normalsize\bfseries}
\titleformat*{\subsubsection}{\normalsize\bfseries}
\titleformat*{\paragraph}{\normalsize\bfseries}
\titleformat*{\subparagraph}{\normalsize\bfseries}

\date{}

\title{Repartition of the quasi-stationary distribution and first exit point density for a double-well potential }

 \author{ Dorian Le Peutrec\thanks{Laboratoire de Math\'ematiques d'Orsay, Univ. Paris-Sud, CNRS, Universit\'e Paris-Saclay, 91405 Orsay, France. E-mail: dorian.lepeutrec@math.u-psud.fr}$\, \ \ $and $\, \ $Boris Nectoux \thanks{Institut f\"{u}r Analysis und Scientific Computing, E101-TU Wien, Wiedner Hauptstr. 8, 1040 Wien, Austria. E-mail: boris.nectoux@asc.tuwien.ac.at} }

\begin{document} 
\maketitle
\begin{abstract}
Let $f: \mathbb R^{d} \to \mathbb R$ be a smooth function and
$(X_t)_{t\ge 0}$ be the stochastic process  solution to the overdamped Langevin dynamics 
$$d X_t = -\nabla f(X_t) d t + \sqrt{h} \ d B_t.$$
Let $\Omega\subset \mathbb R^d$ be a  smooth  bounded domain
and assume
that  $f|_{\overline\Omega}$ is a double-well potential with degenerate barriers.
In this work, we study  in the small temperature regime, i.e. when $h\to 0^{+}$,  the asymptotic repartition  of the quasi-stationary distribution   of $(X_t)_{t\ge 0}$ in $\Omega$ within the two wells of $f|_{\overline\Omega}$. 
We show that this distribution
generically concentrates in precisely one well of  $f|_{\overline\Omega}$ when $h\to 0^{+}$ but can nevertheless concentrate
in both wells 
when $f|_{\overline\Omega}$ admits sufficient symmetries.
This phenomenon corresponds to the so-called  tunneling effect in semiclassical analysis. 
We also investigate in this setting the asymptotic behaviour when $h\to 0^{+}$ of the first exit point  distribution from $\Omega$ of  $(X_t)_{t\ge 0}$ when $X_0$ is distributed according to the quasi-stationary distribution.   
\textbf{AMS classification (2010)}:  35P15,   35P20, 47F05,  35Q82.
\end{abstract}

\section{Setting and results} 

\subsection{Quasi-stationary distribution and purpose of this work}
  \label{sec1}
Let $(X_t)_{t\ge0}$ be the stochastic process solution to the  overdamped Langevin dynamics in~$\mathbb R^d$:
\begin{equation}\label{eq.langevin}
d X_t = -\nabla f(X_t) d t + \sqrt{h} \ d B_t,
\end{equation}
where $f: \mathbb R^d\to \mathbb R$ is the potential (chosen $C^\infty$ in all this work),  $h>0$ is the temperature and $(B_t)_{t\geq 0}$ is a standard $d$-dimensional Brownian motion. Let $\Omega$ be a $C^\infty$ bounded open and connected  subset of $\mathbb R^d$ and introduce 
 $$
\tau_{\Omega} =\inf \{ t\geq 0\,  | \, X_t \notin \Omega     \}$$
  the first exit time from $\Omega$.  
  A quasi-stationary distribution for the  process~\eqref{eq.langevin} 
 on $\Omega$ is a probability measure  $\mu_h$ on $\Omega$ such that, when $X_0$
 is distributed according to  $\mu_h$, what we will denote in the following by 
$X_0 \sim \mu_h$, it holds  for any time $t>0$ and any Borel set $A\subset \Omega$,  
 $$\mathbb P(X_t\in A\,  | \,   t<\tau_\Omega) =\mu_h(A).$$
From~\cite{champagnat2017general,gong1988killed,le2012mathematical,pinsky1985convergence}, there exists a probability measure $\nu_h$ supported in $\Omega$ such that 
  for any probability measure $\mu_0$  on $\Omega$: when $X_0\sim \mu_0$, one has   for any borel set $A\subset \Omega$,
\begin{equation}\label{conv}
\lim_{t\to +\infty} \mathbb P(X_t\in A\,  | \,   t<\tau_\Omega) = \nu_h(A).
\end{equation}  
It follows from \eqref{conv} that $\nu_h$ is the unique quasi-stationary distribution for the  process~\eqref{eq.langevin} 
 on~$\Omega$.\medskip
 
 \noindent 
In molecular dynamics, the quasi-stationary distribution~$\nu_h$ is   used to quantify  the   metastability of the subdomain $\Omega$ of $\mathbb R^d$ as follows: for a probability measure~$\mu_0$ supported in $\Omega$, 
 the domain $\Omega$ is said to be metastable for the initial condition~$\mu_0$ if, 
  when $X_0\sim \mu_0$,  the convergence in~\eqref{conv} is much
quicker than the average  exit time  from~$\Omega$.
When~$\Omega$ is metastable, it is thus relevant to study the exit event  $(\tau_\Omega,X_{\tau_\Omega})$  of the process~\eqref{eq.langevin} from~$\Omega$ starting from~$\nu_h$, i.e. when $X_0\sim \nu_h$. 
This is used in several algorithms aiming at accelarating 
the sampling of the exit even from a metastable domain, 
see
for instance \cite{ArLe,sorensen-voter-00,le2012mathematical,perez2015long}. The study of  the metastability is a very active
field of science research which
 is at the heart of the
numerical challenges observed in molecular dynamics. We refer in particular to \cite{lelievre2016partial}
for an overview on this topic.\medskip 

\noindent
In this work, we study  the repartition when $h\to 0$ of  the  quasi-stationary distribution~$\nu_h$ within the wells of a double-well Morse
potential~$f$ with degenerate barriers  (see the assumption \textbf{[H-Well]} below).    
We show in particular that~$\nu_{h}$ generically concentrates in one well
(see Theorem~\ref{th.main} below) but can also concentrate
in both wells  when the function $f$ is (nearly) even (see Theorems~\ref{th.main2} and \ref{pr.main} below).
According to the analysis led in \cite{DLLN-saddle1} (see also the preprint~\cite{di-gesu-lelievre-le-peutrec-nectoux-Schuss} which concatenates the results of~\cite{DLLN-saddle1} and of \cite{LLN-saddle2}),
the second phenomenon can only appear when the potential function $f$ admits degenerate deepest barriers. 
It is particularly unstable (see Remark~\ref{re.unstable} below) and arises from a strong tunneling effect between the wells.
The asymptotic behaviour of the  law of $X_{\tau_\Omega}$ when $h\to 0$ is also investigated in order to  discuss the metastability of $\Omega$ for deterministic initial conditions within the wells.


\subsection{Connections with the existing literature}
  \label{sec1}

As it will be clearly stated below in the first part of Section~\ref{sec:main-results},  the quasi-distribution $\nu_{h}$ is completely characterized
by the ground state of the Dirichlet realization
of  the infinitesimal generator $L_{f,h}^{(0)}$  of the diffusion~\eqref{eq.langevin},
\begin{align*}
L_{f,h}^{(0)} \  =\    -\frac h2 \Delta +\nabla f\cdot \nabla \ =\   \frac{1}{2h}\,e^{-\frac f h} \,\Delta_{f,h}^{(0)}\, e^{\frac fh},
\end{align*}
where $\Delta_{f,h}^{(0)}= -h^{2}\Delta + |\nabla f|^{2}-h\Delta f$ is the usual Witten Laplacian acting on functions.
In this respect, the techniques used in this work originate from the semiclassical  literature dealing with the obtention of sharp asymptotics 
on the low spectrum of $\Delta_{f,h}^{(0)}$ in the limit $h\to 0$
and we refer in particular in this direction to \cite{HKN} in the case without boundary and
to \cite{HeNi1} in the case of Dirichlet boundary conditions (see also the prior  related works
\cite{BEGK,BGK} 
using potential-theoretic methods and which motivated \cite{HKN,HeNi1}). However, these references
focus on the low spectrum of $\Delta_{f,h}^{(0)}$ and not really on the concentration of the
corresponding eigenfunctions. In addition, though they consider
multiple-well Morse potentials, they do not consider the case of degenerate  barriers.
The case of general Morse potentials $f $, allowing in particular 
degenerate barriers, has nevertheless
  been recently
treated in the case without boundary in \cite{michel2017small} (see also \cite{BD15} for related results) using the techniques of \cite{HKN,HeNi1}.\medskip

\noindent  More closely related to the present work, the already mentioned paper
\cite{DLLN-saddle1} involving both authors generalizes in particular the results of 
\cite{HeNi1} to more general multiple-well Morse potentials but actually focuses on where
 the quasi-stationary distribution (or equivalently the ground state) concentrates in $\Omega$
and  where  the exit point distribution concentrates on $\pa\Omega$. Moreover, 
our results heavily rely on intermediate results
proven in  \cite{DLLN-saddle1} (see Propositions~\ref{pr.Schuss}  and \ref{co.bdy-estim} in Section~\ref{sub.proof-pr1}).
However, the degenerate situation considered in the present paper is excluded
in \cite{DLLN-saddle1}, 
where the principal barrier of $f$ is assumed to be   non degenerate (see indeed \cite[Assumption (A1)]{DLLN-saddle1}).

%
%

\subsection{Double-well potential}
\label{sec:doublewell}
We assume more generally from now on that $\overline\Omega=\Omega\cup \pa \Omega$ is a $C^{\infty}$ oriented compact and  connected Riemannian manifold of dimension $d$ with boundary~$\partial \Omega$. 
The basic assumption in this work is the following:
\begin{itemize}
\item[] \textbf{[H-Well]}:  The function $f$ belongs to $ C^\infty( \overline \Omega , \mathbb R)$, $\vert \nabla f\vert \neq 0$ on $\pa \Omega$, and  $f:\overline \Omega\to \mathbb R$  and $f|_{\pa \Omega}$ are Morse functions. Moreover, the function $f$ has only two local minima $x_1$ and $x_2$ in $\Omega$ which satisfy 
$$ \argmin_{ \overline \Omega}f =\argmin_{ \Omega}f =\{x_1,x_2\}.$$
Finally, the open set $\{x\in\overline\Omega\,,\ f(x)<\min_{\pa \Omega}f\}$ has precisely two connected components, denoted by $\ft C_1$ and $\ft C_2$, such that for all $j\in\{1,2\}$,
$$x_j\in \ft C_j \ \ \text{ and } \  \ \pa \ft C_j\cap \pa \Omega \neq \emptyset.$$
\end{itemize}
\noindent
Under the assumption \textbf{[H-Well]},
the potential function $f$ has precisely two wells, namely
 the open sets $\ft C_1$ and $\ft C_2$. This double-well potential
 is moreover said to have degenerate barriers 
 since the depths of $\ft C_1$ and $\ft C_2$   are the same and equal 
 (see Figure \ref{fig:well-fig})
 \begin{equation}
 \label{eq.H}
H := \min_{\pa \Omega}f-\min_{\overline\Omega} f =  \min_{\pa \Omega}f-\min_{\Omega} f>0.
 \end{equation} 
 Let us also recall that a function~$g: \overline \Omega \to \mathbb R$ is a  Morse function if all its  critical points are non degenerate. This implies in particular that $g$ has a finite number of critical points. 
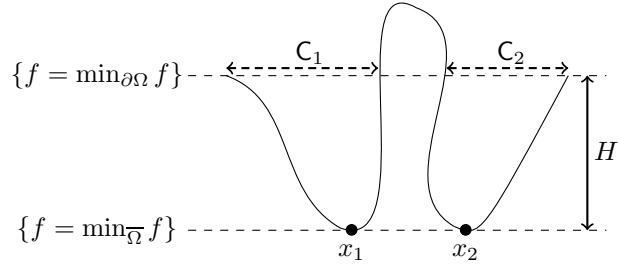
\begin{figure}[h!]
\begin{center}
\begin{tikzpicture}[scale=0.5]
\coordinate (b1) at (0,4);
\coordinate (b2) at (3,-0.04);
\coordinate (b3) at (5,5.8);
\coordinate (b4) at (6,0);
  \draw (b4) ..controls (6.7,-0.3)   .. (9,4) ;
\draw [black!100, in=150, out=-10, tension=10.1]
  (b1)[out=-20]   to  (b2)  to (b3) to (b4);
  \draw [dashed,-]   (-1,4) -- (10,4) ;
   \draw [dashed,-]   (-1,-0.1) -- (10,-0.1) ;
   \draw (-3.4,4) node[]{$ \{f=\min_{\pa \Omega} f\}$};
   \draw (-3.4,-0.1) node[]{$ \{f=\min_{\overline\Omega} f\}$};
    \draw [thick,<->]   (9.5,-0.1) -- (9.5,4) ;
     \draw (10,2) node[]{$H$};
   \draw [thick, densely dashed,<->]   (0,4.2) -- (4,4.2) ;
     \draw (2.2,4.6) node[]{$ \ft C_{1}$};
     \draw [thick, densely dashed,<->]   (5.8,4.2) -- (9,4.2) ;
      \draw (7.5,4.6) node[]{$ \ft C_{2}$};
 
    \tikzstyle{vertex}=[draw,circle,fill=black,minimum size=4pt,inner sep=0pt]
\draw (3.3,-0.09) node[vertex,label=south: {$x_1$}](v){}; 
\draw (6.3,-0.09) node[vertex,label=south: {$x_2$}](v){}; 
    \end{tikzpicture}
\caption{\small{A one dimensional example where \textbf{[H-Well]} is satisfied.}}
 \label{fig:well-fig}
 \end{center}
\end{figure}

\noindent
When replacing the assumption $\argmin_{ \Omega}f =\{x_1,x_2\}$ by $\argmin_{\Omega}f =\{x_1\}$ in \textbf{[H-Well]} (i.e. when the  barriers are not degenerate), it is proved in~\cite[Proposition~9]{DLLN-saddle1} that the quasi-stationary distribution $\nu_h$ concentrates in $\ft C_1$ when $h\to 0$. 
This work aims precisely at studying the degenerate case $\argmin_{ \Omega}f =\{x_1,x_2\}$
which introduces some additional technical difficulties, see the next section for some explanation.
\medskip

\noindent
 Let us assume from now on that the assumption~\textbf{[H-Well]} is satisfed.
The set of saddle points  of~$f$ of index~$1$   in~$\Omega$  is denoted by~$\ft U_1^{\Omega}$. Let us
also define 
\begin{align*}
\ft U_1^{\pa \Omega}\ :\!&=\  \{z \in \pa \Omega, \,  z \text{ is a local minimum of } f|_{\pa \Omega} \}\cap \{z\in \pa \Omega ,\, \pa_nf(z)>0\}
\end{align*}
and
$$
\ft U_1^{\overline \Omega}:=  \ft U_1^{\pa \Omega} \cup \ft U_1^{ \Omega} \ \ \,  \text{ and } \, \ \ \ft m_1^{\overline \Omega}:={\rm Card}(\ft U_1^{\overline \Omega}).$$
According to the terminology of~\cite[Section 5.2]{HeNi1}, 
we call the elements  of $\ft U_1^{\overline \Omega}$ 
the  generalized saddle points for  the Witten Laplacian acting on $1$-forms  with tangential Dirichlet boundary conditions on $\pa \Omega$.
Note that $f$ does not have  any saddle point
on $\pa\Omega$ (since $\nabla f\neq 0$ there)
but that extending $f$ by $-\infty$ outside $\overline\Omega$
(which is consistent with zero boundary Dirichlet conditions),
the elements of $\ft U_1^{\overline \Omega}$ are geometrically saddle points
(since for such an element $z$, $z$ is a local minimum of $f|_{\pa\Omega}$
and a local maximum of $f|_{D}$, where $D$ is the straight line passing through 
$z$ and orthogonal to $\pa\Omega$ at $z$).\medskip
  
\noindent
Notice that from the assumption
\textbf{[H-Well]}, one has for all $ i\in \{1,2\}$:
$$\pa \ft C_i\cap \pa \Omega\  \subset\  \ft U_1^{\pa \Omega} \cap \argmin_{\pa \Omega}f
\ =\ 
(\pa \ft C_1\cup \pa \ft C_2)\cap \pa \Omega \,.
$$
Let us define, for $ i\in   \{1,2\}$, $z_{i,1},\dots,z_{i,\ft n_i}$ by
 \begin{equation} \label{eq.ni}
 \pa \ft C_i\cap \pa \Omega  = \{z_{i,1},\dots,z_{i,\ft n_i}\}, \ \ \text{ where }  \ft n_i \ge 1 \text{ according to \textbf{[H-Well]}}. 
   \end{equation}
One defines furthermore $z_{3,1},\dots,z_{3,\ft n_3}$  by 
$$\{z_{3,1},\dots,z_{3,\ft n_3}\}=\ft U_1^{\overline \Omega}  \setminus \big(\cup_{j=1}^2 \pa \ft C_j\cap \pa \Omega\big ) \,,$$
where $\ft n_3 \in \mathbb N$ ($\ft n_3=0$ meaning $\ft U_1^{\overline \Omega}  \setminus \big(\cup_{j=1}^2 \pa \ft C_j\cap \pa \Omega\big )=\emptyset $).
From~\cite[Proposition~15]{DLLN-saddle1}, it holds
$$\pa \ft C_1\cap\pa \ft C_2 \subset \ft U_1^{\Omega}\cap \{f=\min_{\pa \Omega}f\}$$
and one orders  $z_{3,1},\dots,z_{3,\ft n_3}$ so that
$$\pa \ft C_1\cap\pa \ft C_2 = \{z_{3,1},\dots,z_{3,\ft m_3}\}  \,, $$
where  $\ft m_3 \in \{0,\dots,\ft n_3\}$.
Note finally the relation
\begin{equation}
\label{eq.m1}
\ft m_1^{\overline \Omega}= \ft n_1+\ft n_2+ \ft n_3.
\end{equation}
See Figures~\ref{fig:shema_nota} and \ref{fig:shema_nota2} for a schematic representation of the potential $f$ under \textbf{[H-Well]} when $\pa \ft C_1\cap\pa \ft C_2 =\emptyset$
and when $\pa \ft C_1\cap\pa \ft C_2 \neq\emptyset$.

\begin{figure}[h!]
\begin{center}
\begin{tikzpicture}[scale=0.55]
\tikzstyle{vertex}=[draw,circle,fill=black,minimum size=4pt,inner sep=0pt]
\tikzstyle{ball}=[circle, dashed, minimum size=1cm, draw]
\tikzstyle{point}=[circle, fill, minimum size=.01cm, draw]
\draw [rounded corners=10pt] (1,0.5) -- (-0.25,2.5) -- (1,5) -- (5,6.5) -- (7.6,3.75) -- (6,1) -- (4,0) -- (2,0) --cycle;
\draw [thick, densely dashed,rounded corners=5pt] (2,0.5) -- (.25,1.5) -- (0.5,2.5) -- (0.09,3.5) -- (2.75,3.75) -- (3.5,3) -- (2.4,2) -- (3.44,1.5) --cycle;
\draw [thick, densely dashed,rounded corners=5pt]    (2.75,3.9)  -- (3.4,4) -- (4,2) --(5.5,2.5)--(6.5,4) --(6.5,5)  -- (5,6) -- (3,5)  --cycle;
 \draw (2,1.5) node[]{$\ft C_1$};
  \draw (5,3) node[]{$\ft C_2$};
     \draw  (4.4,1.5) node[]{$\Omega$};
    \draw  (6.9,1.5) node[]{$\pa \Omega$};

\draw (0.17,3.4) node[vertex,label=north west: {$z_{1,2}$}](v){};
\draw (2,2.5) node[vertex,label=north : {$x_1$}](v){};
\draw (4.6 ,4.3) node[vertex,label=north: {$x_2$}](v){};
\draw (0.38,1.45) node[vertex,label=south west: {$z_{1,1}$}](v){};
\draw (6.2,5.2) node[vertex,label=north east: {$z_{2,1}$}](v){};
\end{tikzpicture}

%
%
%
%

\caption{\small{Schematic representation of the connected components $\ft C_1$ and $\ft C_2$ of~$\{f<\min_{\partial \Omega}f\}$  when the assumption  \textbf{[H-Well]} is  satisfied. In this representation,~$\pa{\ft C_1} \cap \pa{ \ft C_2}=\emptyset$, $\ft U_1^{\pa \Omega}=\{z_{1,1},z_{1,2},z_{2,1}\}$,~$ \pa \ft C_1\cap \pa \Omega=\{z_{1,1},z_{1,2}\}$,~$ \pa \ft C_2\cap \pa \Omega=\{z_{2,1}\}$, $\ft U_1^{ \Omega}=\emptyset$ and $\argmin_{\overline \Omega}f=\{x_1,x_2\}$. Thus, $\ft n_1=2$,  $\ft n_2=1$ and $\ft n_3=\ft m_3=0$.}   }
 \label{fig:shema_nota}
 \end{center}
\end{figure}
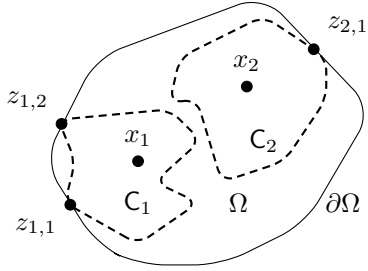

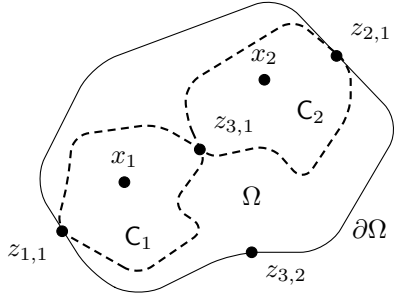
\begin{figure}[h!]
\begin{center}
\begin{tikzpicture}[scale=0.62]
\tikzstyle{vertex}=[draw,circle,fill=black,minimum size=4pt,inner sep=0pt]
\tikzstyle{ball}=[circle, dashed, minimum size=1cm, draw]
\tikzstyle{point}=[circle, fill, minimum size=.01cm, draw]
\draw [rounded corners=10pt] (1,0.5) -- (-0.25,2.5) -- (1,5) -- (5,6.5) -- (7.6,3.75) -- (6,1) -- (4,1) -- (2,0) --cycle;
\draw [thick, densely dashed,rounded corners=6pt] (2,0.5) -- (.25,1.5) -- (0.5,2.5) -- (0.6,3.5) -- (2.4,3.75) -- (3.5,3.1) -- (2.7,2) -- (3.44,1.5) --cycle;
\draw [thick, densely dashed,rounded corners=6pt]    (3,3.9)  -- (3.4,3) -- (4.8,3.3) --(5.5,2.5)--(6.5,4) --(6.5,5)  -- (5.3,6) -- (3,4.58)  --cycle;
 \draw (2,1.3) node[]{$\ft C_1$};
  \draw (5.65,4) node[]{$\ft C_2$};
     \draw  (4.4,2.2) node[]{$\Omega$};
    \draw  (6.9,1.5) node[]{$\pa \Omega$};
\draw (3.3 ,3.2) node[vertex,label=north east: {$z_{3,1}$}](v){};
\draw  (4.4,1) node[vertex,label=south east: {$z_{3,2}$}](v){};
\draw (1.7 ,2.5) node[vertex,label=north  : {$x_1$}](v){};
\draw (4.67 ,4.69) node[vertex,label=north: {$x_2$}](v){};
\draw (0.38,1.45) node[vertex,label=south west: {$z_{1,1}$}](v){};
\draw (6.2,5.2) node[vertex,label=north east: {$z_{2,1}$}](v){};
\end{tikzpicture}

%
%
%
%

\caption{\small{Schematic representation of the connected components $\ft C_1$ and $\ft C_2$ of~$\{f<\min_{\partial \Omega}f\}$  when the assumption   \textbf{[H-Well]} is satisfied.  In this representation,   $\pa{\ft C_1} \cap \pa{ \ft C_2}=\{z_{3,1}\}$, $\ft U_1^{\pa \Omega}=\{z_{1,1},z_{2,1},z_{3,2}\}$,~$ \pa \ft C_1\cap \pa \Omega=\{z_{1,1}\}$,~$ \pa \ft C_2\cap \pa \Omega=\{z_{2,1}\}$, $\ft U_1^{ \Omega}=\{z_{3,1}\}$ and $\argmin_{\overline \Omega}f=\{x_1,x_2\}$. Thus, $\ft n_1=1$, $\ft n_2=1$, $\ft m_3=1$ and $\ft n_3=2$.}}
 \label{fig:shema_nota2}
 \end{center}
\end{figure}

 \subsection{Results}
 \label{sec:main-results}

\noindent{\bf Preliminary spectral analysis}

\noindent
 Let $L_{f,h}^{(0)}$ be  the infinitesimal generator  of the diffusion~\eqref{eq.langevin},
$$
L_{f,h}^{(0)}= \frac h2 \Delta^{(0)}_H +\nabla f\cdot \nabla,
$$
where $ \Delta^{(0)}_H$ is the Hodge Laplacian on $\overline\Omega$ and $\nabla$ the gradient associated with the metric tensor on~$\overline\Omega$.  
Let moreover $L_{f,h}^{D,(0)}$  be the differential operator $L_{f,h}^{(0)}$   on  $L^2(\Omega, e^{-\frac 2h f(x)  }dx)$  with domain 
$$
D\big( L_{f,h}^{D,(0)}\big)=\big \{ w \in H^2(\Omega,e^{-\frac 2h f(x)}dx ), \ w=0 \text{ on }   \pa \Omega   \big   \}.
$$
 The operator $L^{D,(0)}_{f,h}$
 is self-adjoint, positive, and has compact resolvent.  Moreover, its smallest eigenvalue $\lambda_{1}(h)$ is positive, non degenerate, and any eigenfunction associated with~$\lambda_{1}(h)$ has a sign on $\Omega$ (see for instance~\cite[Section~6]{Eva}). Let $u_h$ be an eigenfunction associated with $\lambda_{1}(h)$. 
According to \cite{le2012mathematical}, the quasi-stationary distribution $\nu_h$ is then given  by
\begin{equation}\label{nuh}
d\nu_h:=\frac{ u_h (x)\   e^{-\frac 2h f(x)} }{ \displaystyle \int_{\Omega}   u_h \ e^{-\frac 2h f} }dx,
\end{equation}
where $dx$ is the Lebesgue measure on $\Omega$. We assume furthermore
from now on that  
\begin{equation}\label{uh.norma}
u_h>0 \text{ on } \Omega \text{ and } \int_\Omega   u_h  ^2\, e^{-\frac 2h f}=1.
\end{equation}
 In view of~\eqref{nuh}, 
 in order to study the
 asymptotic behaviour of $\nu_h$  when $h\to 0$,
 we look for an accurate approximation of  $u_h$. 
 This is delicate since exponentially small eigenvalues of the same order are into play. Indeed, according
 to  \cite[Theorem~4]{DLLN-saddle1},
  under \textbf{[H-Well]},
it holds
$$
\lim_{h\to 0} h\ln\big(\lambda_1(h)\big) \ =\ -2(\min_{\pa \Omega}f-\min_{\overline\Omega} f)
\ =\ -2H
$$ 
and  there exists
 $C>1$ such that  for every $h>0$ small enough,
 $$1 < \frac{\lambda_2(h)}{\lambda_{1}(h)}   \le C,$$
 where $ \lambda_2(h)$ denotes the second smallest eigenvalue of 
  $L_{f,h}^{D,(0)}$.
 This makes in particular difficult to properly estimate $u_{h}$
 by simply projecting a well chosen quasi-mode on $\text{Span} (u_h)$
 since the quality of such an approximation is typically
 bounded from above by the quotient $\frac{ \lambda_1(h) }{\lambda_{2}(h)}$
 which does not tend to $0$ when $h\to 0$.
 To overcome this difficulty, 
 the key point  relies on the fact that we are able to 
 precisely 
 analyse the
  restriction of~$L^{D,(0)}_{f,h}$ to the eigenspace associated with $\lambda_{1}(h)$ and $\lambda_2(h) $.
  Indeed,
  this eigenspace has dimension two
  and the remaining eigenvalues of $L_{f,h}^{D,(0)}$
are bounded from below by $\frac{\sqrt h}2$\footnote{They are actually bounded from below by some positive constant.}.  More precisely, we have according to \cite[Theorem 3.2.3]{HeNi1} the
\begin{lemma} \label{ran1}
Let us assume that the  hypothesis \textbf{[H-Well]} is satisfied. Then, there exists $h_0>0$ such that for all $h\in (0,h_0)$,
$$
\dim \Ran \, \pi_{[0,\frac{\sqrt h}2)}\left (L^{D,(0)}_{f,h}\right )= 2,$$
where $\pi_{[0,\frac{\sqrt h}2)}\left (L^{D,(0)}_{f,h}\right )$ is the orthogonal projector on the vector space associated with the eigenvalues of $L^{D,(0)}_{f,h}$ in $[0,\frac{\sqrt h}2)$. 
\end{lemma}
\begin{remark}\label{re.mu}
As a consequence of Lemma~\ref{ran1}, there exists $h_0>0$ such that for every $h\in (0,h_0)$,  the second smallest eigenvalue $\lambda_2(h)$ of $L^{D,(0)}_{f,h}$ is non degenerate. 
\end{remark}

\noindent
Moreover,
it follows  from the general analysis led in \cite{DLLN-saddle1} that
 the matrix $L$ of~$L^{D,(0)}_{f,h}\big |_{       \Ran \, \pi_{[0,\frac{\sqrt h}2)}\left (L^{D,(0)}_{f,h}\right )}$
satisfies Proposition~\ref{pr1} below. Before stating it, 
let us introduce the following notation. 
For $(\gamma(h))_{h>0} \in \mathbb R^{\mathbb R^*_+}$, one writes $\gamma(h)\eqsim \sqrt h$ 
 if there exist $C>0$ and $h_0>0$ such that for all $h\in (0,h_0)$,
 \begin{equation}\label{sim}
 \frac 1C \,  \sqrt{h}  \le \gamma(h) \le C \,  \sqrt{h}.
 \end{equation}
In addition,  for $\alpha>0$, one says that $(r(h))_{h>0}$ admits a full asymptotic expansion in~$h^\alpha$, and one writes $r(h)\sim \sum\limits_{k=0}^{+\infty} a_kh^{\alpha k}$,    if there exists a sequence $(a_k)_{k\geq 0}\in \mathbb R^{\mathbb N}$ such that for any $N\in \mathbb N$, it holds in the limit $h\to 0$: 
 \begin{equation}\label{expension-def}
r(h)=\sum_{k=0}^Na_kh^{\alpha k}+\mathcal O\big (h^{\alpha(N+1)}\big ).
\end{equation}

\begin{proposition}\label{pr1}
Let us assume that the hypothesis \textbf{[H-Well]} is satisfied.  Then, there exists $h_0>0$ such that for every $h\in (0,h_0)$, there exists an orthonormal basis $\mathcal B_0=(\varphi_{1},\varphi_{2})$ of $\Ran \, \pi_{[0,\frac{\sqrt h}2)}\left (L^{D,(0)}_{f,h}\right )$ such that  the matrix $L$ of the restriction of $L^{D,(0)}_{f,h}$ to $       \Ran \, \pi_{[0,\frac{\sqrt h}2)}\left (L^{D,(0)}_{f,h}\right )$ in~$\mathcal B_0$ has  the form:
\begin{equation}
\label{L}
L=\frac 12 \ \begin{pmatrix}
\alpha_1(h) & \ve(h) \\
\ve(h) & \alpha_2(h)
\end{pmatrix} \, \, h^{-\frac 12} \ e^{-\frac 2h H} ,
\end{equation}
where $H$ is defined in \eqref{eq.H},
\begin{itemize}
\item~$\ve(h)$ satisfies in the limit $h\to 0$:
\begin{equation}
\label{veh}
\ve(h) = \left\{
    \begin{array}{ll}
        \mathcal O\big( e^{-\frac ch}\big) & \text{ if  } \pa{\ft C_1} \cap \pa{ \ft C_2}=\emptyset  \\
        \eqsim \sqrt h  &  \text{ if  }  \pa{\ft C_1} \cap \pa{ \ft C_2}\neq \emptyset,
    \end{array}
\right.
\end{equation}
for some $c>0$ independent of $h$ and where the symbol $\eqsim$ is defined in~\eqref{sim}, 
\item there exist two sequences~$(\kappa_{1,k})_{k\geq 0}\in \mathbb R^{\mathbb N}$ and $(\kappa_{2,k})_{k\geq 0}\in \mathbb R^{\mathbb N}$ such that for $i\in \{1,2\}$, in the limit $h\to 0$:
\begin{equation}
\label{alpha1}
\alpha_i(h) \sim \left\{
    \begin{array}{ll}
        \sum \limits_{k=0}^{+\infty} \kappa_{i,k}h^{k} & \text{ if  } \pa{\ft C_1} \cap \pa{ \ft C_2}=\emptyset  \\
        \sum \limits_{k=0}^{+\infty} \kappa_{i,k}h^{\frac k2} &  \text{ if  }  \pa{\ft C_1} \cap \pa{ \ft C_2}\neq \emptyset,
    \end{array}
\right.
\end{equation}
 where the symbol $\sim$ is defined in~\eqref{expension-def} and 
\begin{equation}
\label{aj}
\kappa_{i,0}=\sum_{j=1}^{\ft n_i} \frac{ 2\,  \pa_nf(z_{i,j})}{\pi^{\frac 12}} \, \frac{ \sqrt{ {\rm det \ Hess } f   (x_{i}) }  }{  \sqrt{ {\rm det \ Hess } f|_{\partial \Omega}   (z_{i,j}) }  }\, .
\end{equation}
Moreover, when $\pa{\ft C_1} \cap \pa{ \ft C_2}\neq \emptyset$, one has for every $i\in\{1,2\}$,
\begin{equation}
\label{aj-12}
\kappa_{i,1}= \sum_{j=1}^{\ft m_3}  
 \frac{|\lambda_{-}(z_{3,j})|  \, \big(\det \Hess f(x_{i})\big)^{\frac12}}{\pi\, \big| \det \Hess f(z_{3,j})   \big|^{\frac12}} ,
\end{equation}
where $\lambda_{-}(z)$ is the negative eigenvalue of $\Hess f(z)$. 
Finally, the sequence $(\kappa_{1,k})_{k\geq 1}$ (resp. $(\kappa_{2,k})_{k\geq 1}$)  only  depends on the values of the derivatives of $f$ at $x_1$ and on  $\pa \ft C_1\cap \big (\pa \Omega\cup \pa \ft C_2\big)$ (resp.  of  the derivatives of $f$ at $x_2$ and on  $\pa \ft C_2\cap \big (\pa \Omega\cup \pa \ft C_1\big)$). 

\end{itemize}

\end{proposition}

\noindent Proposition~\ref{pr1} will be proven in Section~\ref{sec.21}.
It permits to reduce  the study of  the asymptotic repartition of $\nu_h$ within the wells $\ft C_1$ and $\ft C_2$  to  linear algebra considerations in dimension two. 
Then,  when $X_0\sim \nu_h$, the study of the asymptotic concentration of the law of $X_{\tau_\Omega}$
(which occurs on a subset of $\argmin_{\pa \Omega}f$, see \cite[Definition~1]{DLLN-saddle1} for a precise definition) follows from the analysis made in~\cite{DLLN-saddle1} and based on the following formula  \cite{le2012mathematical}: for any  $F\in L^{\infty}(\partial \Omega,\mathbb R)$, it holds
\begin{equation}\label{eq.loi-bord}
\mathbb E^{\nu_h} \left [ F(X_{\tau_{\Omega}} )\right]=- \frac{h}{2\lambda_1(h)}  \frac{\displaystyle\int_{\partial \Omega}F\,\partial_n u_h  \ e^{-\frac{2}{h}  f}  }{\displaystyle\int_\Omega u_h e^{-\frac{2}{h}  f}},
\end{equation} 
where the notation $\mathbb E^{\nu_h}$ stands for the expectation when $X_0\sim \nu_h$.\medskip

\noindent{\bf Results when~$\nu_h$ concentrates in precisely one well when $h\to 0$}
 

\noindent
Let us define here the following assumption:
  \begin{itemize}
\item[]\textbf{[H1]}: The assumption  \textbf{[H-Well]} is satisfied,
there exists $h_0>0$ such that 
\begin{align} \label{h1a}
\text{either}\ \ \  &\text{for all $h\in (0,h_0)$, }  \alpha_1(h)<\alpha_2(h),\\
\label{h1b}
\text{or}\ \  \ &\text{for all $h\in (0,h_0)$, }  \alpha_2(h)<\alpha_1(h),
\end{align} 
and it holds 
$$\lim_{h\to 0}\  \frac{\ve(h)}{\alpha_1(h)-\alpha_2(h)}=0.  $$     
\end{itemize}
Note that the assumption \textbf{[H1]} is generic (given an arbitrary function $f$ satisfying \textbf{[H-Well]})
according to the following:
    \begin{itemize}
\item    when $\pa \ft C_1\cap \pa \ft C_2=\emptyset$ and the asymptotic expansion in $h$ of $\alpha_1(h)$ and $\alpha_2(h)$ in~\eqref{alpha1} differ (i.e. when $(\kappa_{1,k})_{k\geq 0}\neq (\kappa_{2,k})_{k\geq 0}$),
the assumption  \textbf{[H1]} is satisfied and there exists $c>0$ such that when $h\to 0$ (see indeed  \eqref{veh}),
\begin{equation} 
\label{h1-ca}
 \frac{\ve(h)}{\alpha_1(h)-\alpha_2(h)}=\mathcal O\big (e^{-\frac ch}\big ),
 \end{equation}

\item  when $\pa \ft C_1\cap \pa \ft C_2\neq \emptyset$
the assumption  \textbf{[H1]} is, according to \eqref{veh} and \eqref{alpha1}, equivalent to  $\kappa_{1,0}\neq \kappa_{2,0}$, where $\kappa_{1,0}$ and $\kappa_{2,0}$ are defined in~\eqref{aj}. In this case,  when $h\to 0$:
\begin{equation} 
\label{h1-cb}
 \frac{\ve(h)}{\alpha_1(h)-\alpha_2(h)}\eqsim\mathcal O\big (\sqrt h\big ).
 \end{equation} 
    \end{itemize}
    \noindent
   Our main result under the generic assumption \textbf{[H1]} is the following.
It implies in particular that when \textbf{[H1]} holds together with \eqref{h1a},  
$\nu_h$ concentrates in any neighborhood  of $x_1$ (i.e. $\lim_{h\to 0} \nu_h(\ft O_1)=1$ for any open subset~$\ft O_1$ of $\Omega$ containing $x_1$, see more precisely~\eqref{eq.QSD-x1} below).
This can be roughly explained as follows: when \textbf{[H1]} holds,
the term  $\ve(h)$ can be neglected in the expression
 of the matrix $L$ given  in \eqref{L}, and 
 \eqref{h1a}  breaks the symmetry  between the two wells $\ft C_1$ and $\ft C_2$,
 ensuring more precisely the concentration of $\nu_{h}$ in $\ft C_1$. 
       
\begin{theorem}
\label{th.main}
Let us assume that the hypotheses \textbf{[H-Well]} and \textbf{[H1]} together with
\eqref{h1a} are satisfied. Let $\nu_h$ be the quasi-stationary distribution of the process~\eqref{eq.langevin} on $\Omega$ (see~\eqref{nuh}). 
Let  $\ft O_{1}\subset \Omega$ be an open neighborhood  of $x_{1}$ and $\ft O_{2}\subset \Omega$   be an open neighborhood  of $x_{2}$ such that $\ft O_1\cap \ft O_2=\emptyset$. Then, there exists $c>0$ such that   in the limit $h\to0$:
\begin{equation}
\label{eq.QSD-x1-0}
 \nu_h(\ft O_{1})+  \nu_h(\ft O_{2})=1
+\mathcal O\big (e^{-\frac ch}\big ),
\end{equation}
where for $k\in \{1,2\}$,
\begin{equation}
\label{eq.QSD-x1}
\nu_h(\ft O_{k})=
\delta_{1,k}+ \mathcal O\left(\frac{|\varepsilon(h)|}{|\alpha_2(h)-\alpha_1(h)|}\right) +\mathcal O\big (e^{-\frac ch}\big ).
\end{equation}
\begin{sloppypar}
\noindent
Moreover, for any $F\in L^{\infty}(\partial \Omega,\mathbb R)$ and for any family $(\Sigma_{i,j})_{(i,j)\in\bigcup_{p=1}^2 \{p\}\times \{1,\dots, \ft n_{p}\}}$ of  disjoint  open neighborhoods of $(z_{i,j})_{(i,j)\in\bigcup_{p=1}^2 \{p\}\times \{1,\dots, \ft n_{p}\}}$ in $\pa\Omega$,
there exists $c>0$ such that in the limit  $h\to 0$:
\begin{equation} \label{t-1}
\mathbb E^{\nu_h} \left [ F\left (X_{\tau_{\Omega}} \right )\right]=\sum \limits_{(i,j)\in\bigcup_{p=1}^2 \{p\}\times \{1,\dots, \ft n_{p}\}}
\mathbb E^{\nu_h} \left [ \mathbf{1}_{\Sigma_{i,j}}F\left (X_{\tau_{\Omega}} \right )\right]  +\mathcal O\big (e^{-\frac ch}\big  )
 \end{equation}
 and
 \begin{equation} \label{t-2}
\sum \limits_{j=1}^{\ft n_{2}}
\mathbb E^{\nu_h} \left [ \mathbf{1}_{\Sigma_{2,j}}F\left (X_{\tau_{\Omega}} \right )\right]  =
 \mathcal O\left(\frac{|\varepsilon(h)|}{|\alpha_2(h)-\alpha_1(h)|}\right) +\mathcal O\big (e^{-\frac ch}\big ). \end{equation}
In addition, when, for some $j\in\{1,\dots,\ft n_{1}\}$, $F$ is $C^{\infty}$ around $z_{1,j}$, one has when $h\to 0$:
\begin{equation}\label{t-3}
 \mathbb E^{{\nu_h}} \left [ \mathbf{1}_{\Sigma_{1,j}}F\left (X_{\tau_{\Omega}} \right )\right]=F(z_{1,j})\,a_{1,j} + \mathcal O\left(\frac{|\varepsilon(h)|}{|\alpha_2(h)-\alpha_1(h)|}\right)+
\mathcal O\big (h\big ),\end{equation}
 where,  for $i\in \{1,2\}$ and $j\in \{1,\dots,\ft n_i\}$,  the constant $a_{i,j}$ is defined by 
 \begin{equation} \label{eq.ai}
  a_{i,j}:=\frac{  \partial_nf(z_{i,j})      }{  \sqrt{ \det \Hess f\big|_{\partial \Omega}   (z_{i,j}) }  } \left (\sum \limits_{k=1}^{\ft n_{i}} \frac{  \partial_nf(z_{i,k})      }{  \sqrt{ \det \Hess f\big|_{\partial \Omega}   (z_{i,k}) }  }\right)^{-1}. 
  \end{equation}
 \end{sloppypar}
\end{theorem}

\begin{remark}
When  \textbf{[H-Well]} and \textbf{[H1]} are satisfied, one also obtains
from Proposition~\ref{pr1}
sharp asymptotic estimates on the two smallest eigenvalues $0<\lambda_1(h)<\lambda_2(h)$ of $L_{f,h}^{D,(0)}$ when $h\to 0$, see indeed~\eqref{eq.vp1} and~\eqref{eq.vp1'}. 
\end{remark}

\noindent 
From Theorem~\ref{th.main},  when \textbf{[H-Well]} holds  and \textbf{[H1]} is satisfied with~\eqref{h1a}, the quasi-stationary distribution $\nu_h$   concentrates  
when $h\to 0$ in   $\ft C_1$ and more precisely around any arbitrary small neighborhood of $x_{1}$. 
Moreover, when $X_0\sim \nu_h$,
 the law of $X_{\tau_\Omega}$ concentrates  when $h\to 0$ on 
 $\{z_{1,1},\dots,z_{1,\ft n_{1}}\}=\pa\ft  C_{1}\cap\pa\Omega$
with an  explicit repartition given by~\eqref{eq.ai}.
Adapting the proof of \cite[Proposition~11]{di-gesu-lelievre-le-peutrec-nectoux-Schuss} (see also~\cite{LLN-saddle2})
by using \eqref{eq.QSD-x1-0} and \eqref{eq.QSD-x1},
one can also show that when $X_0=x\in \ft C_1$, the 
law of $X_{\tau_\Omega}$ concentrates  when $h\to 0$ on 
 $\{z_{1,1},\dots,z_{1,\ft n_{1}}\}=\pa\ft  C_{1}\cap\pa\Omega$ with the  same repartition
as when $X_{0}\sim\nu_{h}$. This exhibits a metastable behavior  for such 
  initial conditions. 
 Moreover, when  $\vert \nabla f\vert \neq 0$ on $ \pa \ft C_2$, it follows from
 \cite[Theorem~2]{di-gesu-lelievre-le-peutrec-nectoux-Schuss} 
 that  when $X_0=x\in \ft C_2$, the 
law of $X_{\tau_\Omega}$ concentrates  when $h\to 0$ on 
 $\{z_{1,2},\dots,z_{1,\ft n_{2}}\}=\pa\ft  C_{2}\cap\pa\Omega$ with the repartition
 given by~\eqref{eq.ai} (with $i=2$). This exhibits a non metastable behavior for such 
  initial conditions.  
 \medskip
 
 \noindent 
   To connect with the literature dealing with semiclassical Schr\"odinger operators of the form $h^2\Delta_H^{(0)}+V$ on manifolds without boundary (where $V$ is a potential function independent of $h$),
one can say in this situation that the
 tunneling effect between the two wells is too weak to
mix their respective properties and that these two wells are hence somehow independent,
that is, in the terminology of \cite{HelSjII,HelSjIII}, weakly resonant or non resonant.
We also refer  to 
\cite{He2}  for an overview on this topic for semiclassical Schr\"odinger operators
(see in particular pp.~41--42 there).
Notice lastly that \eqref{eq.QSD-x1}  
shows that some tunneling effect of order
$\sqrt h$ appears nevertheless
when $\pa\ft  C_{1}\cap \pa \ft C_{2}\neq\emptyset$ (see indeed \eqref{h1-cb}),
contrary to the case
$\pa \ft C_{1}\cap \pa \ft C_{2}=\emptyset$ when 
$\alpha_1(h)$ and $\alpha_2(h)$ do not have the same asymptotic expansion, see \eqref{h1-ca}. As expected, when $\pa\ft  C_{1}\cap \pa \ft C_{2}\neq\emptyset$, the independence between the two
wells in this case is hence generically weaker.\medskip

%
%

\noindent
{\bf  Results when $\nu_h$ concentrates in both wells when $h\to 0$}

\noindent   
  Let us define here the following assumption:
  \begin{itemize}
\item[] \textbf{[H2]}: The assumption  \textbf{[H-Well]} is satisfied. Moreover,  there exists $h_0>0$ such that for all $h\in (0,h_0)$, it holds $$
\ve(h)\neq 0\ \ \text{and}\ \ \lim_{h\to 0}\  \frac{ \alpha_1(h)-\alpha_2(h)}{\ve(h)}=0.  
$$
\end{itemize}

\noindent
Let us exhibit situations where the assumption  \textbf{[H2]} is satisfied.
\begin{itemize}
\item When  $\pa \ft C_1\cap \pa \ft C_2\neq \emptyset$, the assumption \textbf{[H2]} is satisfied if and only if $\kappa_{1,0}=\kappa_{2,0}$ and $\kappa_{1,1}=\kappa_{2,1}$. This equivalence follows from~\eqref{veh} and~\eqref{alpha1}. Therefore, when $\pa \ft C_1\cap \pa \ft C_2\neq \emptyset$, using \eqref{aj} and \eqref{aj-12}, the assumption \textbf{[H2]} is satisfied if and only if
\begin{equation}
\label{eq1}
\sum_{j=1}^{\ft n_1}   \frac{\pa_nf(z_{1,j})   }{  \sqrt{ {\rm det \ Hess } f|_{\partial \Omega}   (z_{1,j}) }  }=\sum_{j=1}^{\ft n_2}   \, \frac{ \pa_nf(z_{2,j}) }{  \sqrt{ {\rm det \ Hess } f|_{\partial \Omega}   (z_{2,j}) }  },
\end{equation}
and
\begin{equation}
\label{eq2}
{\rm det \ Hess } f   (x_{1}) ={\rm det \ Hess } f   (x_{2}). 
\end{equation}
Moreover, it holds in this case:
\begin{equation*} 
\label{h1-cc}
\frac{ \alpha_1(h)-\alpha_2(h)}{\ve(h)}=\mathcal O\big (\sqrt h\big ).
 \end{equation*} 
\item  Let us assume that $f$  is an even function as defined by~\eqref{eq.sym} below. Then,  from Theorem~\ref{pr.main} below, the assumption \textbf{[H2]} is satisfied (see indeed Remark~\ref{re.sym1}). 

\end{itemize}
\begin{remark}
When  $\pa \ft C_1\cap \pa \ft C_2= \emptyset$, we are not able  to explicit  assumptions  on $f$ which imply \textbf{[H2]} except in the symmetric situation described 
in
Theorem~\ref{pr.main}. 
Note in particular that
when $\pa \ft C_1\cap \pa \ft C_2= \emptyset$ and  \textbf{[H2]} holds, one has  when $h\to 0$:  $\alpha_1(h)=\alpha_2(h)\, \big (1+O\big (e^{-\frac ch}\big )\big )$ (which follows from \textbf{[H2]},~\eqref{alpha1} and the fact that $\ve(h)=O\big (e^{-\frac ch}\big )$, see~\eqref{veh}) and thus:
\begin{equation}
\label{eq3}
\kappa_{1,k}=\kappa_{2,k}\  \text{ for all } k\in \mathbb N.
\end{equation}
Moreover, it also holds in this case
$\lambda_{1}(h)=\lambda_2(h)\big (1+ \mathcal O\big(e^{-\frac ch})\big)$
(see \eqref{eq.l1=l2}).
\end{remark}

\begin{remark}
\label{re.unstable}
The assumption \textbf{[H2]} is non generic, that is unstable with respect  to perturbations of the potential~$f$ in the following sense. For any $f$ satisfying  \textbf{[H2]}, it follows from \eqref{eq1}--\eqref{eq3} that there exists an arbitrary small perturbation $\delta f:\overline \Omega \to \mathbb R$ such that $f+\delta f$ satisfies \textbf{[H1]}. Then, according to Theorem~\ref{th.main},   the quasi-stationary distribution
for the potential $f+\delta f$
  concentrates when $h\to 0$  in precisely one of the wells $\ft C_1$ or $\ft C_2$.
\end{remark}

\noindent
The following result shows that when  \textbf{[H2]} is satisfied, the quasi-stationary distribution~$\nu_h$ concentrates when $h\to 0$  in the two wells~$\ft C_1$ and~$\ft C_2$. 

\begin{theorem}
\label{th.main2} Let us assume that the hypotheses \textbf{[H-Well]} and \textbf{[H2]} are satisfied.  Let $\nu_h$ be the quasi-stationary distribution of the process~\eqref{eq.langevin} on $\Omega$ (see~\eqref{nuh}). 
 Let  $\ft O_{1}\subset \Omega$ be an open neighborhood  of $x_{1}$ and $\ft O_{2}\subset \Omega$   be an open neighborhood  of $x_{2}$ such that $\ft O_1\cap\ft O_2=\emptyset$. Then,  there exists $c>0$ such that    in the limit $h\to0$: 
\begin{equation}\label{eq.QSD-1-h2}
\nu_h(\ft O_{1})+  \nu_h(\ft O_{2})=1
+\mathcal O\big (e^{-\frac ch}\big ),
\end{equation}
where, for $k\in\{1,2\}$,  
\begin{equation}\label{eq.QSD-2-h2}
\nu_h(\ft O_{k})
=b_{k}+\mathcal O\left(\frac{|\alpha_2(h)-\alpha_1(h)|}{|\varepsilon(h)|}\right)  +     \mathcal O\big(h \big),
\end{equation}
where, defining $q$ by $\{q\}=\{1,2\}\setminus\{k\}$,
\begin{equation}
\label{eq.bi}
\displaystyle b_{k}=\frac{\displaystyle \big(\det \Hess f(x_{q})\big)^{\frac14}} {\displaystyle\big(\det \Hess f(x_1)\big)^{\frac14}+\big(\det \Hess f(x_2)\big)^{\frac14}}.
\end{equation}
Moreover, for any $F\in L^{\infty}(\partial \Omega,\mathbb R)$ and for any family $(\Sigma_{i,j})_{(i,j)\in\bigcup_{p=1}^2 \{p\}\times \{1,\dots, \ft n_{p}\}}$ of disjoint open neighborhoods of $(z_{i,j})_{(i,j)\in\bigcup_{p=1}^2 \{p\}\times \{1,\dots, \ft n_{p}\}}$ in $\pa\Omega$,
there exists $c>0$ such that in the limit  $h\to 0$:
\begin{equation}\label{eq.h2r}
\mathbb E^{\nu_h} \left [ F\left (X_{\tau_{\Omega}} \right )\right]=\sum \limits_{(i,j)\in\bigcup_{p=1}^2 \{p\}\times \{1,\dots, \ft n_{p}\}}
\mathbb E^{\nu_h} \left [ \mathbf{1}_{\Sigma_{i,j}}F\left (X_{\tau_{\Omega}} \right )\right]  +\mathcal O\big (e^{-\frac ch}\big).
 \end{equation}
Lastly, when, for some $(i,j)\in\bigcup_{p=1}^2 \{p\}\times \{1,\dots, \ft n_{p}\}$, $F$ is $C^{\infty}$ around $z_{i,j}$, one has when $h\to 0$:
\begin{equation}\label{eq.h2r-b}
 \mathbb E^{\nu_h} \left [ \mathbf{1}_{\Sigma_{i,j}}F\left (X_{\tau_{\Omega}} \right )\right]=F(z_{i,j})\,a_{i,j}\,b_{i} +
\mathcal O\left (\frac{|\alpha_2(h)-\alpha_1(h)|}{|\varepsilon(h)|}\right)+
\mathcal O\big(h\big),
\end{equation}
 where $b_{i}$ is defined in \eqref{eq.bi} and $a_{i,j}$ is defined in
 \eqref{eq.ai}.
 
 \end{theorem}
 
 \begin{remark}
When  \textbf{[H-Well]} and \textbf{[H2]} are satisfied, one also gives sharp asymptotic estimates on the two smallest eigenvalues $0<\lambda_1(h)<\lambda_2(h)$ of $L_{f,h}^{D,(0)}$ when $h\to 0$, see indeed \eqref{eq.l1=l2},~\eqref{lh1-h2-c1} and~\eqref{lh1-h2-c2}.
 \end{remark}

\noindent 
When \textbf{[H-Well]}    and \textbf{[H1]} hold, Theorem~\ref{th.main2} implies  that the quasi-stationary distribution~$\nu_h$   concentrates  
when $h\to 0$ in   $\ft C_1$ and~$\ft C_2$,  and more precisely around any arbitrary small neighborhood of $x_{1}$ and $x_2$.  Note also that when $\partial  \ft C_{1}\cap \partial \ft  C_{2}\neq \emptyset$,  the coefficient \eqref{eq.bi}   
specifying the repartition of $\nu_{h}$ within the wells equals $\frac12$
according to \eqref{eq2}.
 Moreover, when $X_0\sim \nu_h$ 
 the law of $X_{\tau_\Omega}$ concentrates  when $h\to 0$ on 
 $\{z_{1,1},\dots,z_{1,\ft n_{1}}\}\cup \{z_{2,1},\dots,z_{2,\ft n_{2}}\}=(\pa\ft  C_{1}\cup \pa\ft  C_{2})\cap\pa\Omega$ 
with an  explicit repartition given by~\eqref{eq.ai}. 
In addition, when   $\vert \nabla f\vert \neq 0$ on $\pa \ft C_1 \cup \pa \ft C_2$, it follows from
 \cite[Theorem~2]{di-gesu-lelievre-le-peutrec-nectoux-Schuss}
 that  when $X_0=x\in \ft C_k$,  $k\in \{1,2\}$, the 
law of $X_{\tau_\Omega}$ concentrates  when $h\to 0$ on 
 $\{z_{k,1},\dots,z_{k,\ft n_{k}}\}=\pa\ft  C_{k}\cap\pa\Omega$ with the repartition
 given by~\eqref{eq.ai}. This shows that in this case the domain $\Omega$ is not metastable for deterministic initial conditions within $ \ft C_1 \cup  \ft C_2$.\medskip

\noindent
 Connecting again with the literature dealing with semiclassical Schr\"odinger
operators of the form $h^2\Delta_H^{(0)}+V$ on manifolds without boundary, when  the assumptions \textbf{[H-Well]} and \textbf{[H2]} are satisfied,  a strong tunneling effect appears when $h\to 0$ and mixes the respective properties of both wells.  
We refer  to \cite[pp.~45--46]{He2} for  a symmetric case  with two wells
and to
 \cite{HelSjII} for more general symmetric situations.\medskip

\noindent
Let us conclude this section by specifying the statement of
Theorem~\ref{th.main2} in a completely symmetric situation. To this end, we recall that an isometry 
$ \Phi:\overline\Omega\to\overline\Omega$ is a $C^\infty$ diffeomorphism   which satisfies, for all  $x\in \overline \Omega$ and all $v,w\in T_x \overline \Omega$: $v\cdot w=D\Phi_x(v)\cdot D\Phi_x(w)$, where $\cdot$ is the scalar product associated with the metric of $\overline\Omega$ on the tangent bundle $T\overline\Omega$. One says moreover that $f:\overline \Omega\to \overline\Omega$ is  even if there exists an isometry $\Phi$ such that
\begin{equation}\label{eq.sym}
\Phi(x_{1})=x_{2},  \ \Phi ^2=I, \  \text{ and } \ f\circ \Phi=f,
\end{equation} 
where $I$ is the identity map on $\overline \Omega$. 
When $f$ is even, the following improvement of Theorem~\ref{th.main2} holds.

\begin{theorem}
\label{pr.main}
Let us assume that the hypothesis  \textbf{[H-Well]} is  satisfied.  Let $\nu_h$ be the quasi-stationary distribution of the process~$(X_t)_{t\ge 0}$ on $\Omega$ (see~\eqref{nuh}). 
Assume that $f$ is an even function as defined by~\eqref{eq.sym}. Then,  the assumption \textbf{[H2]} is satisfied with in particular, for all $h$ small enough: $$\alpha_1(h)=\alpha_2(h), \text{ where } \alpha_1(h)  \text{ and } \alpha_2(h)  \text{ are defined by~\eqref{L}}.$$
Furthermore, let $\ft  O_{1}\subset \Omega$ be an open neighborhood  of $x_{1}$ and $\ft O_{2}\subset \Omega$   be an open neighborhood  of $x_{2}$ such that $\ft O_1\cap \ft O_2=\emptyset$. Then, for $k\in\{1,2\}$, there exists $c>0$ such that  in the limit $h\to0$: 
$$
\nu_h(\ft O_{k})=\frac12+\mathcal 
O\big(e^{-\frac ch}\big).
$$
Moreover, one has $\ft n_{1}=\ft n_{2}$ (see~\eqref{eq.ni}) and the asymptotic estimates~\eqref{eq.h2r} and~\eqref{eq.h2r-b}. 
  \end{theorem}

   \section{Proof of our main results}
   \label{sec.2}

    \subsection{Proof of Proposition~\ref{pr1}}
   \label{sec.21}
   
     \subsubsection{The operator $L_{f,h}^{D,(1)}$}
         \label{sec.nota-sobo}
 For $p\in\{0,\dots,d\}$, one denotes by
$\Lambda^pC^{\infty}(\overline{\Omega})$ the space of $C^{\infty}$ $p$-forms on
$\overline\Omega$ and by  $\Lambda^pC^{\infty}_T(\overline\Omega)$ the subset of
$\Lambda^pC^{\infty}(\overline{\Omega})$ made of the  $p$-forms $v$ such that $\mbf tv=0$ on $\partial \Omega$, where $\mathbf t$ denotes the tangential trace on forms. We recall that $\mbf tv=0$ on $\partial \Omega$
means that the restriction  to $\pa\Omega$ of the  $p$-form $v$ vanishes  when applied to tangential vector fields, and
we refer e.g. to 
 \cite[Equation (2.25)]{GSchw} for a rigorous definition of the tangential trace. 
For  $q\in \mathbb N$, one denotes by
$\Lambda^pH^q_w(\Omega)$ the weighted Sobolev spaces of $p$-forms with
regularity index~$q$, for the weight $e^{-\frac{2}{h} f} $ on
$\Omega$ (where the subscript $w$  refers to the fact that the weight function
appears in the inner product), and we refer again to \cite{GSchw} for an introduction to weighted Sobolev spaces on manifolds with boundaries. The set $\Lambda^pH^1_{w,T}(\Omega)$ is then defined by 
$$\Lambda^pH^1_{w,T}(\Omega):= \left\{v\in \Lambda^pH^1_w(\Omega)
  \,,\   \mathbf  tv=0 \ {\rm on} \ \partial \Omega\right\}.$$
We will denote by $\Vert . \Vert_{H^q_w}$  the norm on the weighted space $\Lambda^pH^q_w
(\Omega)$ and by $\langle
\cdot , \cdot\rangle_{L^2_w}$ the scalar product on $\Lambda^pL^2_w (\Omega)$. Notice that $\Lambda^0L^2_w (\Omega)$ is the space   
  $L^2(\Omega, e^{-\frac 2h f(x)}dx)$  and $\Lambda^0H^2_w  (\Omega)$ is the space   
  $H^2(\Omega,e^{-\frac 2h f(x)}dx )$ introduced in the definition of  $L_{f,h}^{D,(0)}$ in Section~\ref{sec:main-results}.\medskip
   
\noindent
   In the following, one denotes respectively  by $d: \Lambda^{p} \,C^{\infty}(\overline \Omega)\to \Lambda^{p+1} \,C^{\infty}(\overline\Omega)$  and $d^*:\Lambda^{p+1} \,C^{\infty}(\overline\Omega)\to \Lambda^{p} \,C^{\infty}(\overline\Omega)$ the exterior and the co-differential derivatives  on $\overline\Omega$. 
     Let us introduce  the differential operator
   $$L^{(1)}_{f,h}\ =\ \frac{h}{2} \Delta^{(1)}_H + \mathcal L_{\nabla f}
   \ =\ \frac{1}{2h}e^{\frac fh}\big(h^{2}\Delta^{(1)}_H +|\nabla f|^{2} +h (\mathcal L_{\nabla f} +\mathcal L^{*}_{\nabla f}) \big)e^{-\frac fh}$$
  acting on $\Lambda^1C^{\infty}(\overline{\Omega})$, where $\Delta^{(1)}_H=(d+d^*)^2$ is the Hodge Laplacian on $\overline\Omega$,  $\mathcal L_{\nabla f}$ is the Lie derivative with respect to the vector field $\nabla f$, and 
 $\mathcal L^{*}_{\nabla f}$ its formal adjoint in $L^{2}(\Omega)$.
The (tangential) Dirichlet realization of $L^{(1)}_{f,h}$
is denoted by
 $L^{D,(1)}_{f,h}$ and its domain is
 $$D\big (L^{D,(1)}_{f,h}\big)= \big \{ v \in \Lambda^1 H^2(\Omega), \ \mbf{t}v=0 \text{ and }   \mbf{t} d^* e^{-\frac 2h f}v =0 \text{ on }   \pa \Omega\big \}.$$
From \cite[Section~2.4]{HeNi1}, the operator $L^{D,(1)}_{f,h}$
 is self-adjoint, positive and has compact resolvent. 
 One has moreover the following result from \cite[Theorem 3.2.3]{HeNi1}. 
\begin{lemma} \label{ran1b}
Under the assumption \textbf{[H-Well]}, there exists $h_0>0$ such that for all $h\in (0,h_0)$,
$$
\dim \Ran \, \pi_{[0,\frac{\sqrt h}2)}\left (L^{D,(1)}_{f,h}\right )= \ft m_1^{\overline \Omega},$$
where $ \ft m_1^{\overline \Omega}$ is defined in \eqref{eq.m1} and $\pi_{[0,\frac{\sqrt h}2)}\left (L^{D,(1)}_{f,h}\right )$ is the orthogonal projector on the vector space  associated with the eigenvalues of $L^{D,(1)}_{f,h}$ in $[0,\frac{\sqrt h}2)$. 
\end{lemma}

\noindent
In the following, the   exterior differential $d$ will be denoted, with a slight abuse of notation,  by  $\nabla$. For ease of notation, one  also denotes, for $p\in \{0,1\}$,
$$\pi_h^{(p)}=\pi_{[0,\frac{\sqrt h}2)}\left (L^{D,(p)}_{f,h}\right ).$$
From \cite[Corollary~2.4.4]{HeNi1}, the following relation holds on $\Lambda^{0} \,H^{1}_{w,T}(\Omega)$:
\begin{equation}\label{commute2}
 \nabla \pi_h^{(0)} =  \pi_h^{(1)}\nabla\,.
\end{equation}
This  implies in particular that
\begin{equation}\label{nabla}
\nabla: \Ran \,  \pi_h^{(0)}\to \Ran \,  \pi_h^{(1)}
\end{equation}
and then, when \textbf{[H-Well]} holds, according to Lemma~\ref{ran1},  that
 for every   $h$ small enough,
\begin{equation}\label{uhin}
\nabla u_h  \in \Ran \,  \pi_h^{(1)}.
\end{equation}
We refer to \cite[Section~3.1.2]{DLLN-saddle1}
for more details concerning this section.

   \subsubsection{Proof of Proposition~\ref{pr1}}
   \label{sub.proof-pr1}
In the following, we assume that 
\textbf{[H-Well]} holds.\medskip

\noindent
The finite dimensional vector spaces $\Ran\, \pi_h^{(0)}$  and $\Ran 
\, \pi_h^{(1)}$ are endowed with the scalar product $\langle. ,.\rangle_{L^2_w}$ of $L^2_w(\Omega)$ introduced in Section~\ref{sec.nota-sobo}. Moreover, the set $ \{i_j\}_{(i,j)\in\bigcup_{p=1}^3 \{p\}\times \{1,\dots, \ft n_{p}\}} $ is ordered using the lexicographical order, i.e. 
$$ \{i_j\}_{(i,j)\in\bigcup_{p=1}^3 \{p\}\times \{1,\dots, \ft n_{p}\}}=\{1_1,\dots.1_{\ft n_1},2_1,\dots,2_{\ft n_2},3_1\dots,3_{\ft m_3}, 3_{\ft m_3+1},\dots,3_{\ft n_3}\},$$
where we recall $\ft n_1=\text{Card }(\pa \ft C_1\cap \pa \Omega)$, $\ft n_2=\text{Card }\big (\pa \ft C_2\cap \pa \Omega\big)$,  $\ft m_3=\text{Card }\big(\pa \ft C_1\cap \pa \ft C_2\big)$ and $\ft n_3=\text{Card }\big(\ft U_1^{\overline \Omega} \setminus (\cup_{k=1}^2 \pa \ft C_k\cap \pa \Omega)\big )=\ft m_1^{\overline \Omega}-\ft n_1-\ft n_2$  are defined in Section~\ref{sec:doublewell}.\\

\noindent
Let us now define
\begin{equation}
\label{eq.chi1}
\widetilde u_{1}:=\frac{\chi_{1}}{\big \|\chi_{1}\big \|_{L_{w}^{2}}} \in \Lambda^0 H^1_{w,T}\left (\Omega\right ) \ \text{and}\ \ \widetilde u_{2}:=\frac{\chi_{2}}{\big \|\chi_{2}\big \|_{L_{w}^{2}}}\in \Lambda^0 H^1_{w,T}\left (\Omega\right )
\end{equation}
where, for $i\in\{1,2\}$, $0\not\equiv \chi_i\in C^\infty(\Omega,\mathbb R^+)$ is compactly supported in $\Omega$,
$\chi_{1}$ and $\chi_{2}$ have disjoint supports,
 and for some small $\alpha>0$ and $\beta>0$,
$$\supp \chi_{i}\subset \big(\ft C_{i}+B(0,\alpha)\big)\cap \Omega\ \ \text{and}\ \ \chi_{i}=1\ \text{on }\, \ft C_{i}\cap \{f<\min_{\pa\Omega}f-\beta\}.$$
Let us also consider a family of $L^{2}_{w}$-unitary
$1$-forms 
\begin{equation}
\label{eq.psi}
(\widetilde \psi_{i_j})_{(i,j)\in \bigcup_{p=1}^3 \{p\}\times \{1,\dots, \ft n_{p}\}}
\end{equation}
such that, for  $(i,j)\in \bigcup_{p=1}^3 \{p\}\times \{1,\dots, \ft n_{p}\}$, $\widetilde \psi_{i_j}\in  \Lambda^1 H^1_{w,T}\left (\Omega\right )\cap \Lambda^1 C^\infty\left (\overline \Omega\right ) $, and for some small $\delta>0$, 
$\supp \widetilde\psi_{i_j}\subset B(z_{i,j},\delta)\cap \overline\Omega$.\medskip

\noindent
It then holds,  for every  $(k,q)\in \{1,2 \}$, $(i,j)\in \bigcup_{p=1}^3 \{p\}\times \{1,\dots, \ft n_{p}\}$, and $(i',j')\in \bigcup_{p=1}^3 \{p\}\times \{1,\dots, \ft n_{p}\}$ (for $\delta>0$ small enough):
\begin{equation}
\label{eq.ortho}
\lp \widetilde u_{k}, \widetilde u_{q} \rp_{L^2_w} = \delta_{k,q}\ \ \text{and}\ \ 
 \lp \widetilde \psi_{i_j}, \widetilde \psi_{i'_{j'}} \rp_{L^2_w} =  \delta_{i,i'}\delta_{j,j'}\,.
\end{equation}

\noindent
Taking, for every $(i,j)\in \bigcup_{p=1}^3 \{p\}\times \{1,\dots, \ft n_{p}\}$,   $\widetilde \psi_{i_j}$
as a (normalized) truncated principal eigen-$1$-form of a local Witten Laplacian defined
around $z_{i,j}$ with  Dirichlet boundary conditions
\footnote{Actually, when $z_{i,j}\in\pa\Omega$ and $\mathcal V$ denotes its corresponding neighborhood  in $\overline\Omega$, (full) Dirichlet boundary conditions are considered
on $\pa\mathcal V\cap \Omega$ while only tangential Dirichlet boundary conditions are considered
on $\pa\mathcal V\cap \pa\Omega$.}, we obtain the following proposition (see \cite[Section~3.2.2 and Definition~42]{DLLN-saddle1} and references therein for details). It gathers the statements of \cite[Propositions~43 and 47]{DLLN-saddle1}   
which are the starting points of our analysis.

\begin{proposition}
\label{pr.Schuss}
Let us assume that the function $f$ satisfies \textbf{[H-Well]}. 
Then, the  families $(\widetilde u_{1},\widetilde u_{2})$ and 
$(\widetilde \psi_{i_j})_{(i,j)\in \bigcup_{p=1}^3 \{p\}\times \{1,\dots, \ft n_{p}\}}$ defined
in~\eqref{eq.chi1},~\eqref{eq.psi} can be chosen so that
the following estimates hold when $h\to 0$ (where $H$ is defined in~\eqref{eq.H}):
\begin{enumerate}
\item There exists $c>0$ such that: 
\begin{itemize}
\item[a)]  for  every $k\in \{1,2\}$, it holds
 $$  \big\|     (1-\pi_h^{(0)} ) \widetilde u_k \big  \|_{L^2_w}^2\ \leq\  h^{\frac12}\,  \big \|   \nabla  \widetilde u_k \big \|_{L^2_w}^2      \le   e^{-\frac{2}{h}(H-  \frac c 2)},$$
\item[b)]  for every $ i\in\{1,2,3\}$ and   $j\in\{1,\dots,\ft n_{i}\}$, it holds
\begin{equation*}
   \big\|   (1-\pi_h^{(1)} ) \widetilde \psi_{i_j} \big \|_{H^1_w}^2=   \mathcal O \big (e^{-\frac{2c}{h}} \big).
     \end{equation*} 
\end{itemize}     
\item For every $k\in \{1,2 \}$ and  $(i,j)\in \bigcup_{p=1}^3 \{p\}\times \{1,\dots, \ft n_{p}\}$, there exists a real constant  $\ve_{i,j,k}\in\{-1,1\}$
independent of $h$ such that it holds
\begin{equation*}
  \lp       \nabla \widetilde u_k ,    \widetilde \psi_{i_j} \rp_{L^2_w} =\begin{cases} 
  -C_{i,j,k} \ h^{-\frac34}\,  e^{-\frac Hh}   \    \big(  1  +  \mathcal    O(h )   \big)   &  \text{ when } z_{i,j} \in  \pa \ft C_{k}\cap \pa\Omega \\
  \ve_{i,j,k} C_{i,j,k} \ h^{-\frac12}\,  e^{-\frac Hh }   \    \big(  1  +    \mathcal  O(h )   \big)   &  \text{ when } z_{i,j} \in  \pa \ft C_{1}\cap \pa \ft C_{2}\\
 0   &   \text{ else},  \end{cases} 
  \end{equation*}
where the remainder terms $\mathcal O(h)$ admit a full asymptotic expansion in~$h$,~and
\begin{equation}
\label{eq.C}
 C_{i,j,k}=\begin{cases}\pi ^{-\frac{1}{4}} \sqrt{2\, \partial_n f(z_{i,j}) }  \,  \frac{\big(\displaystyle \det \Hess f(x_{k})\big)^{\frac14}}{\big( \displaystyle \det \Hess f |_{ \partial \Omega}(z_{i,j})   \big)^{\frac14}}   & \text{if $z_{i,j}\in  \pa \ft C_{k}\cap \pa\Omega$}\\
 \pi^{-\frac{1}{2}} \sqrt{|\lambda_{-}(z_{i,j})| }\,
 \frac{\big( \displaystyle \det \Hess f(x_{k})\big)^{\frac14}}{\big| \displaystyle \det \Hess f(z_{i,j})   \big|^{\frac14}}
 & \text{if $z_{i,j}\in  \pa \ft C_{1}\cap \pa \ft C_{2}$},
 \end{cases}
\end{equation}
where $\lambda_{-}(z_{i,j})$ denotes the negative eigenvalue of $\Hess f(z_{i,j})$.
  \end{enumerate}
\end{proposition}

\begin{remark}
 \label{Ree.estime1}
In the second item in Proposition~\ref{pr.Schuss}, notice that  it follows from the notation introduced in Section~\ref{sec:doublewell} that for every $k\in\{1,2\}$ and $(i,j)\in \bigcup_{p=1}^3 \{p\}\times \{1,\dots, \ft n_{p}\}$, one has:
\begin{itemize}
\item $
z_{i,j}\in \pa\ft  C_{k}\cap \pa\Omega 
$ if and only if $i=k$ (and thus $j\in\{1,\dots,\ft n_{k}\}$),
\item $z_{i,j}\in \pa \ft C_{1}\cap \pa \ft C_{2}$
if and only if $i=3$ (and thus $j\in \{1,\dots,\ft m_{3}\}$).
\end{itemize}
 \end{remark}

  \noindent
 As a consequence of \eqref{eq.ortho} and the first item in Proposition~\ref{pr.Schuss},
there exists $c>0$ such that it holds in the limit $h\to 0$:
\begin{equation}
 \label{eq.G0}
 G_0:=\left(\big\lp \pi_h^{(0)}  \widetilde  u_{k},  \pi_h^{(0)} \widetilde  u_{q} \big\rp_{L^2_w}\right)_{k,q\in\{1,2\}}= 
 I_{2}+ \mathcal O  \big(e^{-\frac{c}{h}} \big)
\end{equation}
 and  
 \begin{equation}
 \label{eq.G1}
G_1:=\! \left(\big \lp \pi_h^{(1)}  \widetilde  \psi_{i_j},  \pi_h^{(1)}  \widetilde \psi_{i'_{j'}}  \big\rp_{L^2_w}\right)_{ \substack{(i,j)\in \bigcup_{p=1}^3 \{p\}\times \{1,\dots, \ft n_{p}\}\\  \, (i',j')\in \bigcup_{p=1}^3 \{p\}\times \{1,\dots, \ft n_{p}\} } }\!\!= I_{\ft m_1^{\overline \Omega}}+ \mathcal O (e^{-\frac{c}{h}}).\!
\end{equation}
It then follows from Lemmata~\ref{ran1} and~\ref{ran1b} that,
for every $h>0$ small enough,
 the family 
$  \big (  \pi_h^{(0)}\widetilde u_k  \big )_{k\in\{1,2\}}$ is a basis of $\range   \pi_h^{(0)}$ and 
that
$ \big  (\pi_h^{(1)}\widetilde \psi_{i_j} \big  )_{(i,j)\in \bigcup_{p=1}^3 \{p\}\times \{1,\dots, \ft n_{p}\}}$ is a basis of  $\range   \pi_h^{(1)}$.\medskip
  
 \noindent 
  Let us now define the $\ft m_1^{\overline \Omega}\times 2$ matrix
\begin{equation}\label{eq.S}
 S:=\left (  \big \lp\nabla \pi_h^{(0)}\widetilde u_k ,    \pi_h^{(1)}\widetilde \psi_{i_j}  \big \rp_{L^2_w}\right)_{(i,j)\in \bigcup_{p=1}^3 \{p\}\times \{1,\dots, \ft n_{p}\},\, k\in\{1,2\}}.
\end{equation}
According to the two items in Propositions~\ref{pr.Schuss},
and using the identity
$$
 \big \lp\nabla \pi_h^{(0)}\widetilde u_k ,    \pi_h^{(1)}\widetilde \psi_{i_j}  \big \rp_{L^2_w}=
 \big \lp\nabla \widetilde u_k ,    \widetilde \psi_{i_j}  \big \rp_{L^2_w}-
 \big \lp\nabla \widetilde u_k ,  \big (1-\pi_h^{(1)} \big )  \widetilde \psi_{i_j}  \big \rp_{L^2_w}
$$
which follows from~\eqref{commute2}, 
there exists $c>0$ such that the coefficients of $S$ satisfy when $h\to 0$:
\begin{equation}\label{eq.Se}
S_{i_j,k}=\begin{cases}
 \big   \lp       \nabla \widetilde u_k ,    \widetilde \psi_{i_j} \big  \rp_{L^2_w}\big(1+\mathcal O\big (e^{-\frac ch}\big )\big)    
  &  \text{ if } z_{i,j} \in  \pa \ft C_{k}\cap \pa\Omega \\
  \big  \lp       \nabla \widetilde u_k ,    \widetilde \psi_{i_j}  \big \rp_{L^2_w}\big(1+\mathcal O\big (e^{-\frac ch}\big )\big)  
  &  \text{ if } z_{i,j} \in  \pa \ft C_{1}\cap \pa\ft  C_{2}\\
 \mathcal O \big (e^{-\frac 1h (H +c)} \big )   &   \text{ else}.  \end{cases} 
\end{equation}
\noindent
Let us  denote by $\widetilde{ \Upsilon}$ and  $\widetilde{ \Psi}$
 the following families written as row vectors,
$$
\widetilde{ \Upsilon}\ :=\ \big( \pi_h^{(0)}  \widetilde  u_{1}, \pi_h^{(0)}  \widetilde  u_{2}\big) \quad
\text{and}
\quad
\widetilde{ \Psi}\ :=\ \big(\pi_h^{(1)}\widetilde \psi_{i_j}\big)_{(i,j)\in \bigcup_{p=1}^3 \{p\}\times \{1,\dots, \ft n_{p}\}},
$$
and define
\begin{equation}
\label{eq.B0-B1}
\mathcal B_{0}= (\varphi_{1},\varphi_{2}) 
 := \widetilde{ \Upsilon}\,G_{0}^{-\frac12}\ \ 
\text{and}\ \ 
\mathcal B_{1}=  \big (\psi_{i_{j}} \big )_{(i,j)\in \bigcup_{p=1}^3 \{p\}\times \{1,\dots, \ft n_{p}\}} := \widetilde{ \Psi}\,G_{1}^{-\frac12} ,
\end{equation}
where $G_0$ and $G_1$ are defined in \eqref{eq.G0} and \eqref{eq.G1}. For every $h>0$ small enough, the families $\mathcal B_{0}$ and $\mathcal B_{1}$ are then respectively
orthonormal bases of   $\range   \pi_h^{(0)}$ and of $\range   \pi_h^{(1)}$.\\
 The matrix $L$  of $L_{f,h}^{D,(0)}\big|_{\range   \pi_h^{(0)}}$ in the basis $\mathcal B_{0}$
is   given by 
\begin{equation}
\label{eq.L}
L \, =\, G_{0}^{-\frac12} \, \left(  \big \langle L_{f,h}\pi_h^{(0)}  \widetilde  u_{k},\pi_h^{(0)}  \widetilde  u_{q} \big \rangle_{L^2_w}\right)_{1\leq k,q\leq 2}\,G_{0}^{-\frac12}.
\end{equation}
\noindent
This matrix is sometimes called the {\it interaction matrix} in the literature
dealing with the study of semiclassical Schr\"odinger operators
(see e.g. \cite{HelSj1} or \cite{DiSj}).
Moreover, the matrix $M$ of $\nabla :\range   \pi_h^{(0)}\to \range   \pi_h^{(1)}$ (see~\eqref{nabla}) 
in the bases $\mathcal B_{0}$ and $\mathcal B_{1}$
is   given by 
\begin{equation}
\label{eq.M}
M\ =\ G_{1}^{-\frac12}\,S\,G_{0}^{-\frac12},
\end{equation}
where $S$ is defined in \eqref{eq.S}.  Since 
$L_{f,h}^{D,(0)}\big|_{\Ran\, \pi_h^{(0)}}=\frac h2 \nabla^{*}\nabla$,  
the matrix $M$ satisfies
\begin{equation}
\label{eq.ML}
L =\frac h2 M^{*}M.
\end{equation}
In order to prove Proposition~\ref{pr1}, it is then sufficient to get asymptotic estimates on the coefficients of the matrix $M$. This is the purpose of the next proposition. 
\begin{proposition} 
\label{interaction1} 

Let us assume that the hypothesis \textbf{[H-Well]} is satisfied.  Let  $(\widetilde u_{k})_{k\in\{1,2\}}$  be  defined by~\eqref{eq.chi1}. Let $(\varphi_{k})_{k\in\{1,2\}}$  and $ \big (\psi_{i_{j}} \big )_{(i,j)\in \bigcup_{p=1}^3 \{p\}\times \{1,\dots, \ft n_{p}\}}$ be  defined by~\eqref{eq.B0-B1}. Then, for all $k\in\{1,2\}$, there exists $c>0$ such that when $h\to 0$:
\begin{enumerate}
\item[i)] for every   $j\in\{1,\dots,\ft n_{k}\}$,
 \begin{align*}
 \!\!\!\!\! \lp    \nabla   \varphi_k,   \psi_{k_j} \rp_{L^2_w}&=  \langle \nabla  \widetilde u_k, \widetilde \psi_{k_j} \rangle_{L^2_w} \, \big (1+\mathcal O(e^{-\frac ch} )\big) 
  =-C_{k,j,k} \,h^{-\frac34}\,  e^{-\frac Hh}   \, \big (1+   \mathcal  O(h )   \big),
\end{align*} 
\item[ii)] for every   $j\in\{1,\dots,\ft n_{p}\}$ with $p \in\{1,2\}\setminus\{ k\}$,
 \begin{equation*}
  \lp    \nabla   \varphi_k,   \psi_{p_j} \rp_{L^2_w}   =\mathcal O\big(   e^{-\frac 1h(H +c)}   \big),
  \end{equation*}   
\item[iii)]  for every  $j\in\{1,\dots,\ft  m_{3}\}$,
\begin{align*}
\!\!\!\!\!  \lp    \nabla \varphi_k,   \psi_{3_j} \rp_{L^2_w}&=  \langle \nabla  \widetilde u_k, \widetilde \psi_{3_j} \rangle_{L^2_w} \, \big (1+\mathcal O(e^{-\frac ch} )\big) =\varepsilon_{3,j,k}\, C_{3,j,k} \,h^{-\frac12}\,  e^{-\frac Hh }   \, \big (1+ \mathcal    O(h )   \big),
\end{align*}
\item[iv)] and for all $j\in\{ \ft  m_{3}+1,\dots, \ft n_3\}$,
  \begin{equation*}
  \lp    \nabla   \varphi_k,   \psi_{3_j} \rp_{L^2_w}   =\mathcal O\big(   e^{-\frac 1h(H +c)}   \big),
  \end{equation*}   
  \end{enumerate}
  where we recall that $H=\min_{\pa \Omega}f-\min_{\overline\Omega}f$ (see \eqref{eq.H}), the coefficients $C_{i,j,k}$ are defined in \eqref{eq.C},
  and  the terms $\mathcal O(h)$ admit a full asymptotic expansion in $h$.
\end{proposition}

\begin{proof}
The results of Proposition~\ref{interaction1} follow from   \eqref{eq.G0}--\eqref{eq.Se}, \eqref{eq.M}, and item 2 in Proposition~\ref{pr.Schuss} (see also Remark~\ref{Ree.estime1}).
\end{proof}

\noindent
Proposition~\ref{pr1} is a consequence of Proposition~\ref{interaction1} and of \eqref{eq.ML}.
They indeed imply the existence of some $c>0$
such that when $h\to 0$, the coefficients $\ve(h)$, $\alpha_{1}(h)$,
and $\alpha_{2}(h)$
defined by \eqref{L} satisfy
\begin{equation}
\label{veh'}
\ve(h)=\begin{cases}
\mathcal O\big (e^{-\frac ch}\big )& \text{if } \pa \ft C_{1}\cap \ft \pa C_{2}=\emptyset\\
\sum \limits_{j=1}^{\ft m_{3}}\varepsilon_{3,j,1}\,\varepsilon_{3,j,2}\,C_{3,j,1}\, C_{3,j,2}\,\sqrt h\,   \big(  1  +    \mathcal  O(h )   \big) & \text{if } \pa \ft  C_{1}\cap \pa \ft C_{2}\neq\emptyset,
\end{cases}
\end{equation}
and, for $k\in\{1,2\}$,
$$\alpha_k(h)=\begin{cases} \sum \limits_{j=1}^{\ft n_{k}}C^{2}_{k,j,k}    \big(  1  +   \mathcal   O(h )   \big) & \text{if } \pa \ft C_{1}\cap \pa \ft C_{2}=\emptyset\\
\sum \limits_{j=1}^{\ft n_{k}}C^{2}_{k,j,k}    \big(  1  +   \mathcal   O(h )   \big)     +\sum \limits_{j=1}^{\ft m_{3}}C^{2}_{3,j,k}\,\sqrt h\, \big(  1  +  \mathcal    O( h )   \big) & \text{if } \pa \ft C_{1}\cap \pa\ft  C_{2}\neq\emptyset,
\end{cases}
$$
where 
 the $C_{i,j,k}$'s are defined in \eqref{eq.C}
and 
 the remainder terms $\mathcal O(h)$ admit a full
asymptotic expansion in  $h$.
The relations \eqref{veh}--\eqref{aj-12} follow. 
\medskip

\noindent
Let us conclude this section by noticing the following 
consequences of Proposition~\ref{pr1} which will needed
in upcoming computations.\medskip

\noindent
$1.$ From~\eqref{L}, it holds for $i\in\{1,2\}$  and every $h$ small enough:
\begin{equation}
\label{lh1}
\lambda_{i}(h)=\frac{\alpha_1(h)+\alpha_2(h) + (-1)^{i}\sqrt{\big  ( \alpha_2(h) -\alpha_1(h)\big )^2+4 \ve(h)^2     }}{4\sqrt h} e^{-2\frac H h}\,,
\end{equation}
where 
 $0<\lambda_{1}(h)<\lambda_2(h)$ denote the two smallest eigenvalues of $L^{D,(0)}_{f,h}$. 
It then follows from~\eqref{lh1},  \eqref{veh'}, and \eqref{alpha1}  that $4\sqrt h \,  \lambda_{1}(h)\,  e^{2\frac Hh}$ and $4\sqrt h \,  \lambda_2(h)\,  e^{2\frac Hh}$ admit a full asymptotic expansion in $h$ when  $\pa{\ft C_1} \cap \pa{ \ft C_2}=\emptyset$ and in $\sqrt h$ when  $\pa{\ft C_1} \cap \pa{ \ft C_2}\neq\emptyset$.\medskip

\noindent
$2.$
From  \eqref{L}, since $u_h$ is the principal eigenfunction of $L_{f,h}^{D,(0)}$ satisfying \eqref{uh.norma}, one has for any $h>0$ small enough:\medskip

\noindent
-- either $\ve(h)=0$, in which case one has necessarily $\alpha_1(h)\neq \alpha_2(h)$
(since $0<\lambda_{1}(h)<\lambda_2(h)$)
and then 
$$ u_h = \pm \varphi_i\,,$$ 
where 
 the functions $\varphi_1$ and $\varphi_2$ are defined by~\eqref{eq.B0-B1}
 and 
$i\in\{1,2\}$ is such that 
$\alpha_i(h)=\min\big(\alpha_1(h), \alpha_2(h)\big)$,\medskip

\noindent
-- or $\ve(h)\neq 0$, in which case \eqref{lh1} and an elementary computation lead to
\begin{equation}
\label{equa-uh}
  u_h=\pm \left(\frac{1}{\sqrt{1+\beta(h)^2}} \,   \varphi_1 + \frac{\beta(h)}{\sqrt{1+\beta(h)^2}}\, \varphi_2\right),  
    \end{equation}
where $\beta(h)$ is defined by
  \begin{equation}\label{betah}
 \beta(h)=-\frac{2\,\ve(h)}{\alpha_2(h)-\alpha_1(h)+\sqrt{\big  ( \alpha_2(h) -\alpha_1(h)\big )^2+ 4\ve(h)^2     }} .
    \end{equation}

\noindent
We conclude this section by stating the following proposition which will also be needed to study the asymptotic behaviour when $h\to 0$ of the law of $X_{\tau_\Omega}$ when $X_0\sim \nu_h$. 
It is the statement of \cite[Proposition~65]{DLLN-saddle1}
in our specific setting.

\begin{proposition}
\label{co.bdy-estim}
Let us assume that the hypothesis \textbf{[H-Well]} is satisfied.  Let   $ \big (\psi_{i_{j}} \big )_{(i,j)\in\bigcup_{p=1}^3 \{p\}\times \{1,\dots, \ft n_{p}\}}$ be defined by~\eqref{eq.B0-B1}. 
Let $\Sigma$ be an open subset of $\pa\Omega$ and   $F\in L^{\infty}(\partial \Omega,\mathbb R)$. One then has for every $(i,j)\in\bigcup_{p=1}^3 \{p\}\times \{1,\dots, \ft n_{p}\}$,  when $h\to 0$,
\begin{equation*}
  \int_{\Sigma}   F \,   \psi_{i_{j}} \cdot n  \   e^{- \frac{2}{h} f}       =\begin{cases}  \mathcal O \big(h^{\frac{d-3}{4}}     e^{-\frac{1}{h}\min_{\pa \Omega} f } \big)  &  \text{ if } i\in\{1,2\}
  \ \text{and}\ z_{i,j}\in\overline\Sigma  \\
\mathcal O \big( e^{-\frac{1}{h} (\min_{\pa \Omega} f+c)} \big)   &   \text{ if } i=3\ \text{or}\ z_{i,j}\notin\overline\Sigma,
  \end{cases} 
  \end{equation*}
  where the constant $c>0$ is independent of $h$.
 Moreover, when $(i,j)\in\bigcup_{p=1}^2 \{p\}\times \{1,\dots, \ft n_{p}\}$, $z_{i,j}\in {\Sigma}$, and  $F$ is  $C^{\infty}$ around $z_{i,j}$, it holds
  \begin{equation*}
  \int_{\Sigma}   F \,  \psi_{i_{j}} \cdot n  \   e^{- \frac{2}{h} f}       = \pi ^{\frac{d-1}{4}} \frac{\sqrt{2\, \partial_n f(z_{i,j}) }}  {\big( \det \Hess f |_{ \partial \Omega}(z_{i,j})   \big)^{\frac14}}  \, h^{\frac{d-3}{4}}   \,     e^{-\frac{1}{h}\min_{\pa \Omega} f}      \big(  F(z_{i,j})   +   \mathcal   O(h )    \big),
  \end{equation*}
  where the above remainder term $\mathcal O(h)$ admits a full asymptotic expansion in $h$.
\end{proposition}


   \subsection{Proof of Theorem \ref{th.main}}
   \label{se-t1}
   In this Section, one proves Theorem~\ref{th.main}. To this end, let us assume that the  hypotheses  \textbf{[H-Well]} and \textbf{[H1]}, with~\eqref{h1a},  are satisfied. Then, from \eqref{lh1},
 \eqref{veh}, and \eqref{alpha1},  one has in the limit $h\to 0$:
\begin{itemize}
\item when  $\pa \ft C_{1}\cap \pa \ft C_{2}=\emptyset$, it holds $\ve(h)=\mathcal O\big(e^{-\frac ch})$ for some $c>0$ and then, for every $i\in\{1,2\}$,
\begin{equation}
\label{eq.vp1}
2\sqrt h\,e^{\frac 2h H}\, \lambda_{i}(h)
=\alpha_{i}(h)+\mathcal O\big(e^{-\frac ch})
\sim  \sum_{k=0}^{+\infty} \kappa_{i,k} h^k\,,
\end{equation}
\item when  $\pa \ft C_{1}\cap \pa \ft C_{2}\neq \emptyset$, it holds $\ve(h)\eqsim \sqrt h$
and then, for every $i\in\{1,2\}$,
\begin{equation}
\label{eq.vp1'}
2\sqrt h\,e^{\frac 2h H}\, \lambda_{i}(h)
=\alpha_{i}(h)+\mathcal O\big(h)
=\kappa_{i,0}\,    \big(  1  +  \mathcal    O(\sqrt h )   \big),
\end{equation}
where,    the  remainder term $\mathcal O\big (\sqrt h\big )$ in~\eqref{eq.vp1'}  admits  a full 
asymptotic expansion in   $\sqrt h$.
\end{itemize}
Moreover, there exists $h_{0}>0$ such that for all $h\in (0,h_0)$,
   \begin{equation}\label{uh-h10}
   u_h =\pm\left(\frac{1}{\sqrt{1+\beta(h)^2}} \,   \varphi_1 + \frac{\beta(h)}{\sqrt{1+\beta(h)^2}}\, \varphi_2\right),  
   \end{equation}
   where $\beta(h)$ is defined in~\eqref{betah} and $(\varphi_1,\varphi_2)$ is defined in~\eqref{eq.B0-B1} (notice that~\eqref{uh-h10} holds in $H^1_w(\Omega)$).   
   Indeed, this is simply the relation \eqref{equa-uh}
   when $\varepsilon(h)\neq 0$. In addition, when $\varepsilon(h)=0$ and
  $\alpha_{2}(h)>\alpha_{1}(h)$ (the latter relation follows from \eqref{h1a}), it 
 holds 
  $u_{h}=\pm \varphi_{1}$, 
 that is precisely the  relation \eqref{uh-h10}
 since in this case $\beta(h)$ is well defined and 
    $\beta(h)=0$  (see indeed \eqref{betah}).\medskip
   
 \noindent
 Since \textbf{[H1]} implies that
   $\lim_{h\to 0}\  \frac{\ve(h)}{\alpha_1(h)-\alpha_2(h)}=0$,
   one moreover obtains from~\eqref{betah} 
   that
    in the limit $h\to 0$: 
   \begin{equation}\label{beta-h1}
   \beta(h)=\mathcal O\left ( \frac{\vert \ve(h)\vert }{\alpha_2(h)-\alpha_1(h)}  \right).
    \end{equation}  
From~\eqref{uh-h10}, \eqref{beta-h1}
together  with $u_h>0$ on $\Omega$, $\widetilde u_1\geq 0$ on $\Omega$, \eqref{eq.G0}, and \eqref{eq.B0-B1}, 
 one has, for every $h$ small enough: $\big \lp     u_h,   \widetilde u_1  \big \rp_{L^2_w}=  1+o(1)$ and then
       \begin{equation}\label{uh-debut}
  u_h= \big (1+ \mathcal O\left ( \beta(h)^2 \right) \big) \,  \varphi_1\, + \,\mathcal  O\big (\vert\beta(h)\vert\big )\,  \varphi_2.
   \end{equation}
Therefore, using \eqref{eq.G0}, \eqref{eq.B0-B1},   and~\eqref{beta-h1}, there exists $c>0$ such that for every  $h$ small enough:
   \begin{equation}\label{uh-final-1}
     u_h = \left (1+ \mathcal O\big( \beta(h)^2 \big) \right) \,   \widetilde u_1\, + \,\mathcal  O\big (\vert\beta(h)\vert\big )\,  \widetilde u_2\, +\,  \mathcal O\big (e^{-\frac ch}\big )  \ \text{  in } L^2_w(\Omega).
        \end{equation}       
        From \eqref{uh-final-1},  one deduces the following proposition  which implies, using in addition  \eqref{nuh} and~\eqref{beta-h1}, the asymptotic estimates~\eqref{eq.QSD-x1-0} and \eqref{eq.QSD-x1}
        in Theorem~\ref{th.main}.
\begin{proposition}
\label{pr.proj}
Let us assume that the  hypotheses  \textbf{[H-Well]} and \textbf{[H1]} together with
\eqref{h1a} are satisfied. Let  $u_h$ be the principal eigenfunction of $L_{f,h}^{D,(0)}$ satisfying \eqref{uh.norma} and $(\widetilde u_1,\widetilde u_2)$ be the functions  introduced in~\eqref{eq.chi1}.
Then, for every open set $\ft O\subset \Omega$ and  for every $h>0$ small enough:
\begin{enumerate}
\item[i)] When $\ft O\cap \{x_{1},x_{2}\}=\{x_{1}\}$, one has
\begin{align*}
\int  _{\ft O} u_h \ e^{- \frac{2}{h} f } &= \left (1+\mathcal  O\left ( \beta(h)^2 \right) +\mathcal O  \big (e^{-\frac{c}{h} } \big)\right) \, \int  _{\ft O} \widetilde u_1 \ e^{- \frac{2}{h} f } \\
&=\frac{(h \pi)^{\frac{d}{4} }}{\big(\det \Hess f(x_{1})\big)^{\frac14}}e^{-\frac 1h \min_{\overline\Omega}f}
\ \big(1+\mathcal  O\left ( \beta(h)^2 \right) + \mathcal O(h)\big),
\end{align*}
where     $c>0$ is independent of $h$  and $\beta(h)$  satisfies~\eqref{beta-h1}. 
\item[ii)] When $\ft O\cap \{x_{1},x_{2}\} =\{x_{2}\}$,  it holds 
$$
\displaystyle \int  _{\ft O} u_h \ e^{- \frac{2}{h} f } = h^{\frac d4}\,e^{-\frac 1h \min_{\overline\Omega}f} \ \mathcal O\left (\vert \beta(h)\vert+  e^{-\frac{c}{h}} \right)
,$$
 where we recall~$\beta(h)$  satisfies~\eqref{beta-h1} and  $c>0$ is independent of $h$. 
\item[iii)] When $\ft O\cap \{x_{1},x_{2}\} =\emptyset$, it holds 
$$
\displaystyle \int  _{\ft O} u_h \ e^{- \frac{2}{h} f } = \mathcal O\left (e^{-\frac 1h( \min_{\overline\Omega}f+c)}\right ), \, \text{  where $c>0$ is independent of $h$}.
$$  
\end{enumerate}

\end{proposition}
\begin{proof}
The relation \eqref{uh-final-1} leads to
\begin{align*}
\int  _{\ft O}u_h \ e^{-\frac{2}{h} f} = \left  (1+\mathcal  O\left ( \beta(h)^2 \right) \right) \, \int  _{\ft O}   \widetilde u_{1}\ e^{- \frac{2}{h} f }&\, + \, 
\mathcal  O\left (\vert \beta(h) \vert\right)\int  _{\ft O}   \widetilde u_{2}\ e^{- \frac{2}{h} f }\\
&\,\qquad +   \mathcal O\Big (e^{-\frac 1h( \min_{\overline\Omega}f+c)}\Big ),
\end{align*}
where $c>0$ is independent of $h$. 
In addition, one has, for $ i\in\{1,2\}$, $\tilde u_{i}=\frac{\chi_{i}}{\|\chi_{i} \|_{L_{w}^{2}}}$ from \eqref{eq.chi1} and it follows from the  Laplace method that there exists $c>0$ such that for any $k\in\{1,2\}$,  when  $h\to 0$,
  \begin{equation}\label{eq.laplace}
\ \int  _{\ft O}\chi^{k}_{i}\, e^{- \frac{2}{h} f }= \begin{cases} 
\frac{\displaystyle (h \, \pi)^{\frac{d}{2} }}{\displaystyle\big(\det \Hess f(x_{i})\big)^{\frac12}}e^{-\frac 2h\min_{\overline\Omega}f  }\big(1+\mathcal O(h)\big)  &\text{ if } x_{i} \in \ft O \\
\mathcal O\Big (\displaystyle e^{-\frac 2h( \min_{\overline\Omega}f+c)}\Big ) &\text{ if } x_{i} \notin \ft O.
\end{cases}
  \end{equation}
The statement of Proposition~\ref{pr.proj} follows easily. 
\end{proof}

\noindent
We also deduce from \eqref{uh-debut} and 
Proposition~\ref{interaction1} together with  \eqref{beta-h1}
 the following estimates.  

\begin{proposition} \label{lemm2} 
Let us assume that the  hypotheses  \textbf{[H-Well]} and \textbf{[H1]} together with \eqref{h1a}
are satisfied.
Let  $u_h$ be the principal eigenfunction of $L_{f,h}^{D,(0)}$ satisfying \eqref{uh.norma}. Let also~$(\widetilde u_{k})_{k\in\{1,2\}}$ and  $ \big (\widetilde \psi_{i_{j}} \big )_{(i,j)\in\bigcup_{p=1}^3 \{p\}\times \{1,\dots, \ft n_{p}\}}$ be as in  Proposition~\ref{pr.Schuss}, and $ \big (\psi_{i_{j}} \big )_{(i,j)\in\bigcup_{p=1}^3 \{p\}\times \{1,\dots, \ft n_{p}\}}$ be  defined by~\eqref{eq.B0-B1}. Then, there exists $c>0$ such that  in the limit $h\to0$:
\begin{enumerate}
\item[i)] For every  $j\in\{1,\dots,\ft n_{1}\}$,
  \begin{align*}
   \big \lp    \nabla   u_h,   \psi_{1_{j}}  \big \rp_{L^2_w}
   &=-C_{1,j,1} h^{-\frac34}\,  e^{-\frac Hh }   \big (1 +  \mathcal O\left ( \beta(h)^2 \right)+  \mathcal   O(h )   \big),
\end{align*} 
where $C_{1,j,1} $ is defined in~\eqref{eq.C} and~$\beta(h)$  satisfies~\eqref{beta-h1}. 
\item[ii)] For every  $j\in\{1,\dots,\ft  m_{3}\}$, \begin{align*}
   \big \lp    \nabla   u_h,   \psi_{3_j}  \big \rp_{L^2_w}
   &=\mathcal O\big(h^{-\frac12}\,  e^{-\frac Hh}\big),
\end{align*} 
\item[iii)]  When $i=2$ and $j\in\{1,\dots,\ft n_{2}\}$ or, $i=3$ and $j\in \{ \ft  m_{3}+1, \dots, \ft n_3\} $,
  \begin{equation*}
   \big \lp    \nabla    u_h,   \psi_{i_j}  \big \rp_{L^2_w}   =h^{-\frac34}\, e^{-\frac Hh}   \, \mathcal O\left(  \vert\beta(h)\vert+  e^{-\frac ch}  \right ).
  \end{equation*} 
  \end{enumerate}
\end{proposition}

\noindent
We are now in position to prove Theorem~\ref{th.main}.

\begin{proof}[End of the proof Theorem~\ref{th.main}]
To conclude the proof Theorem~\ref{th.main}, it remains to prove~\eqref{t-1},~\eqref{t-2} and~\eqref{t-3}. Let assume that \textbf{[H-Well]} and \textbf{[H1]} hold with~\eqref{h1a} and let us consider $F\in L^\infty(\pa \Omega,\mathbb R)$. 
Let us recall that from \eqref{eq.loi-bord}, one has
$$
 \mathbb E^{\nu_h} \left [ F(X_{\tau_{\Omega}} )\right]=- \frac{h}{2\lambda_1(h)}  \frac{\displaystyle \int_{\partial \Omega}F\,\partial_n u_h  \ e^{-\frac{2}{h}  f}  }{\displaystyle \int_\Omega u_h\,  e^{-\frac{2}{h}  f}}.
$$
Sharp asymptotic estimates when $h\to 0$ of  $\lambda_1(h)$ and  $\displaystyle\int_\Omega u_h\,  e^{-\frac{2}{h}  f}$ are respectively given in~\eqref{eq.vp1}, \eqref{eq.vp1'}  and in Proposition~\ref{pr.proj}. Therefore, to prove~\eqref{t-1},~\eqref{t-2} and~\eqref{t-3}, it only remains to estimate  when $h\to 0$, for an open subset $\Sigma$ of $\pa \Omega$,   the term $\displaystyle\int_{\Sigma}  F \,  \partial_{n}u_h\,  e^{-\frac{2}{h} f}$. 
\medskip

\noindent
Since the family $(\psi_{i_{j}})_{(i,j)\in\bigcup_{p=1}^3 \{p\}\times \{1,\dots, \ft n_{p}\}}$ introduced in~\eqref{eq.B0-B1} is an orthonormal 
basis of  $\range   \pi_h^{(1)}$,  it holds when $h\to 0$,  from the Parseval identity and from 
Propositions~\ref{co.bdy-estim} and \ref{lemm2},
\begin{align}
\label{eq.dec-uh}
\int_{\Sigma}  F \,  \partial_{n}u_h\,  e^{-\frac{2}{h} f}  &= \sum \limits_{(i,j)\in\bigcup_{p=1}^3 \{p\}\times \{1,\dots, \ft n_{p}\}} \big \langle \nabla u_h , \psi_{i_j} \big \rp_{L^2_w}  \int_{\Sigma}  F \,  \psi_{i_j} \cdot n \  e^{- \frac{2}{h}  f}\\
\nonumber
&=\sum \limits_{(i,j)\in\bigcup_{p=1}^2 \{p\}\times \{1,\dots, \ft n_{p}\}}  \big\langle \nabla u_h , \psi_{i_j} \big \rp_{L^2_w}  \int_{\Sigma}  F \,  \psi_{i_j}\cdot n \  e^{- \frac{2}{h}  f}\\ 
&\qquad\qquad\qquad\qquad\qquad\qquad\, +\, \mathcal O \left(e^{-\frac 1h(\min_{\pa \Omega} f+H+c)} \right),
\end{align}
for some $c>0$ independent of $h$. 
When $\overline{\Sigma}$ does not contain any of the $z_{i,j}$'s for $(i,j)\in\bigcup_{p=1}^2 \{p\}\times \{1,\dots, \ft n_{p}\}$,
one has, using again Propositions~\ref{co.bdy-estim} and \ref{lemm2}:
\begin{equation}
\label{eq.bdy1}
\int_{\Sigma} F\,\partial_{n}u_h   e^{-\frac{2}{h} f} = \mathcal O \left(e^{-\frac 1h(\min_{\pa \Omega} f+H+c)} \right),
\end{equation}
\noindent
where  $c>0$ is independent of $h$.\\  Assume now that  $\overline{\Sigma}$ does not contain any of the $z_{1,j}$'s for $j\in\{1,\dots,\ft n_{1}\}$.
One then has in the limit $h\to 0$, using Propositions~\ref{co.bdy-estim} and \ref{lemm2} and defining $H':=\min_{\pa \Omega} f+H$:
\begin{align}
\nonumber
\int_{\Sigma} F\, \partial_{n}u_h\,   e^{-\frac{2}{h} f} &= \sum \limits_{j=1}^{\ft n_{1}}\mathcal O\left(h^{-\frac 34}  e^{-\frac Hh  }\right)\mathcal  O \left(e^{-\frac 1h (\min_{ \pa\Omega} f+c)} \right) +\mathcal O \left(e^{-\frac 1h(H'+c)} \right)
\\ 
\nonumber
&\qquad+\sum_{j=1}^{\ft n_{2}}\mathcal O\left( h^{-\frac{3}{4}}\, \vert \beta(h)\vert\,   e^{-\frac Hh }   \right )
\mathcal  O\left( h^{\frac{d-3}4} e^{-\frac 1h \min_{\pa \Omega} f} \right)\\
\label{eq.bdy2}
&=\mathcal O \left(e^{-\frac 1h(H'+c)} \right)+\mathcal  O\left( h^{\frac{d-6}{4}}  \vert\beta(h)\vert   e^{-\frac {H'}h} \right ),
\end{align}
where  $c>0$ is independent of $h$.\\ 
Finally, let us assume  that $ \Sigma\cap\{z_{1,1},\dots,z_{1,\ft n_{1}}\}=\{z_{1,j}\}$ and  $F$ is $C^{\infty}$ around~$z_{1,j}$. One then has, 
using~\eqref{eq.C}, Propositions~\ref{co.bdy-estim} and \ref{lemm2}:
\begin{align}
\nonumber
\int_{\Sigma} F\, \partial_{n}u_h\,   e^{-\frac{2}{h} f} &= \big\langle \nabla u_h , \psi_{1_j} \big\rp_{L^2_w} \int_{\Sigma}  F \,  \psi_{1_{j}} \cdot n \  e^{- \frac{2}{h}  f} +
h^{\frac{d-6}{4}} e^{-\frac {H'}h}\, \mathcal O\left(  \vert  \beta(h)\vert + e^{-\frac ch} \right )\\
\nonumber
&=
 -    \frac{2\,  \pi ^{\frac{d-2}{4}}\, \partial_n f(z_{1,j})\big(\det \Hess f(x_{1})\big)^{\frac14}}{\big( \det \Hess f |_{ \partial \Omega}(z_{1,j})   \big)^{\frac12}} \, h^{\frac{d-6}{4}}\,  e^{-\frac {H'}h}    \\
 \label{eq.bdy3} 
 &\quad \times  \left (F(z_{1,j})+\mathcal O\big (\vert\beta(h)\vert  +h )   \right).
\end{align}
The estimates~\eqref{t-1},~\eqref{t-2} and~\eqref{t-3},  follows from 
 \eqref{eq.loi-bord} and
\eqref{eq.bdy1}--\eqref{eq.bdy3}, using in addition \eqref{beta-h1}, 
\eqref{eq.vp1}, \eqref{eq.vp1'}, and Proposition~\ref{pr.proj}.  This concludes the proof of Theorem~\ref{th.main}.  
\medskip
   \end{proof}

   \subsection{Proofs of Theorems~\ref{th.main2} and~\ref{pr.main}}
 Let us assume in this section that the  hypotheses  \textbf{[H-Well]} and \textbf{[H2]} are satisfied. We recall that \textbf{[H2]}  means that  there exists $h_0>0$ such that for all $h\in (0,h_0)$, 
 \begin{equation}
\label{h2-lim}
\ve(h)\neq 0\ \ \text{and}\ \ \lim_{h\to 0}\  \frac{\alpha_1(h)-\alpha_2(h)}{\ve(h)}=0.
\end{equation}
We then deduce from~\eqref{lh1} the following:     
\begin{itemize}
\item when  $\pa \ft C_{1}\cap \pa \ft C_{2}=\emptyset$, using in addition~\eqref{alpha1} and the fact that $\ve(h)=\mathcal O\big(e^{-\frac ch})$ for some $c>0$ (see~\eqref{veh}), it holds when $h\to 0$:
\begin{equation}
\label{eq.l1=l2}
\lambda_{1}(h)=\lambda_2(h)\big (1+ \mathcal O\big(e^{-\frac ch})\big),
\end{equation}
and 
\begin{equation}
\label{lh1-h2-c1}
2\sqrt h\,e^{\frac 2h H}\,\lambda_{1}(h)\sim  \sum_{k=0}^{+\infty} \kappa_{1,k} h^k,
\end{equation}
\item when  $\pa \ft C_{1}\cap \pa \ft C_{2}\neq \emptyset$, using  in addition~\eqref{alpha1}, the fact that $\ve(h) \eqsim \sqrt h$ (see~\eqref{veh}) and $\kappa_{1,0}=\kappa_{2,0}$ (see~\eqref{eq1},~\eqref{eq2} and~\eqref{aj}) it holds for $i\in\{1,2\}$ when $h\to 0$:
\begin{equation}
\label{lh1-h2-c2}
\lambda_{i}(h)=\kappa_{1,0} \frac{\displaystyle  e^{-\frac 2h  H }}{2\sqrt h}\,    \big(  1  +  \mathcal    O(\sqrt h )   \big),
\end{equation}
where,    the  remainder term $\mathcal O\big (\sqrt h\big )$ in~\eqref{lh1-h2-c2}  admits  a full 
asymptotic expansion in   $\sqrt h$.
\end{itemize}

 \begin{remark}
\label{re.sym1}
When there exists an isometry
$ \Phi:\overline\Omega\to\overline\Omega$ 
satisfying \eqref{eq.sym}, i.e. such that $\Phi(x_{1})=x_{2}$, $f\circ \Phi=f$, and $\Phi^{2}=I$,
it necessarily holds $\ft n_{1}=\ft n_{2}$ and $\Phi(\{z_{1,1},\dots,z_{1,\ft n_{1}}\})=\{z_{2,1},\dots,z_{2,\ft n_{2}}\}$. 
For every $h>0$ small enough,
it follows moreover from the simplicity of the eigenvalues $\lambda_{1}(h)$ and $\lambda_{2}(h)$
(see Remark~\ref{re.mu}) and from
the positivity of $u_{h}$ in $\Omega$ that $u_{h}\circ \Phi=u_{h}$ and $u_{2,h}\circ \Phi=-u_{2,h}$,
where $u_{2,h}$ denotes any eigenvector of $L_{f,h}^{D,(0)}$ associated with $\lambda_{2}(h)$. 
In addition,
one can choose $\chi_{1}$ and $\chi_{2}$ such that $\chi_{2}=\chi_{1}\circ \Phi$ in \eqref{eq.chi1}. This leads, for $h$ small enough,
to $\pi_h^{(0)}\widetilde  u_{1}+\pi_h^{(0)}\widetilde  u_{2}\in \sspan(u_{h})$, 
$\pi_h^{(0)}\widetilde  u_{1}-\pi_h^{(0)}\widetilde  u_{2}\in \sspan(u_{2,h})$
and hence to 
$$
\lp \pi_h^{(0)}  \widetilde  u_{1},  \pi_h^{(0)} \widetilde  u_{1} \rp_{L^2_w}=
\lp \pi_h^{(0)}  \widetilde  u_{2},  \pi_h^{(0)} \widetilde  u_{2} \rp_{L^2_w}
$$
and
$$
\lp L_{f,h}^{D,(0)}\pi_h^{(0)}  \widetilde  u_{1},  \pi_h^{(0)} \widetilde  u_{1} \rp_{L^2_w}=
\lp L_{f,h}^{D,(0)}\pi_h^{(0)}  \widetilde  u_{2},  \pi_h^{(0)} \widetilde  u_{2} \rp_{L^2_w}.
$$
It then follows from \eqref{eq.G0}, \eqref{eq.L}, and \eqref{L} that for $h$ small enough,
$\alpha_{1}(h)=\alpha_{2}(h)$ and hence, using $\lambda_{1}(h)\neq \lambda_{2}(h)$, that $\varepsilon(h)\neq 0$.
   The relation \eqref{h2-lim} is thus in particular satisfied in this situation. 
\end{remark}

\noindent
Moreover, there exists $h>0$ such that for all $h\in (0,h_0)$,
   \begin{equation}\label{uh-h2}
   u_h =\pm\left (\frac{1}{\sqrt{1+\beta(h)^2}} \,   \varphi_1 + \frac{\beta(h)}{\sqrt{1+\beta(h)^2}}\, \varphi_2\right),  
   \end{equation}
   where $\beta(h)$ is defined in~\eqref{betah} and $(\varphi_1,\varphi_2)$ is defined in~\eqref{eq.B0-B1}.  This
  is indeed simply  \eqref{equa-uh} since $\varepsilon(h)\neq 0$ according to \eqref{h2-lim}. 
 Using~\eqref{h2-lim} 
  and \eqref{betah}, one obtains moreover that when $h\to 0$:
      \begin{equation}\label{beta-h2}
   \beta(h)=-\frac{\vert \ve(h)\vert}{\ve(h)}+\mathcal  O\left ( \frac{\vert \alpha_2(h)-\alpha_1(h)\vert}{\vert \ve(h)\vert }  \right). 
   \end{equation}  
From~\eqref{uh-h2}, \eqref{beta-h2}
together  with $u_h>0$ on $\Omega$, $\widetilde u_1,\widetilde u_2\geq 0$  on $\Omega$, \eqref{eq.G0}, and \eqref{eq.B0-B1}, 
 one has, for every $h$ small enough, $\big \lp     u_h,   \widetilde u_1  \big \rp_{L^2_w}=  \frac{1}{\sqrt 2}+o(1)$
 and $ 0<\big \lp     u_h,   \widetilde u_2  \big \rp_{L^2_w}= - \frac{\vert \ve(h)\vert}{\sqrt2\, \ve(h)}+o(1)$.
 It follows that for every $h$ small enough:
$\ve(h)<0$,
\begin{equation}\label{beta-h2-b}
   \beta(h)=1+ \mu(h)\,,  \text{ where } \, \mu(h)=\mathcal  O\left ( \frac{\vert \alpha_2(h)-\alpha_1(h)\vert}{\vert \ve(h)\vert }  \right) \to 0 \text{ when } h \to 0,
   \end{equation} 
   and 
\begin{equation}
\label{uh-debut2}
 u_h =\ \frac{1}{\sqrt2}\big (1+\mathcal  O\left ( \vert \mu(h)\vert \right) \big) \,  \varphi_1\, + \,  \frac{1}{\sqrt2}\big (1+\mathcal  O\left (  \vert \mu(h)\vert \right) \big) \,  \varphi_2.
\end{equation}
Moreover, using \eqref{eq.G0}, \eqref{eq.B0-B1},    and~\eqref{beta-h2-b}, the equality \eqref{uh-debut2} implies that there exists $c>0$ such that  for every $h$ small enough,
 \begin{align}
\nonumber
  u_h &=\ \frac{1}{\sqrt2}\left  (1+ \mathcal O\left (  \vert \mu(h)\vert  \right) +\mathcal O\big (e^{-\frac ch}\big )\right) \,  \widetilde u_1\, + \,  \frac{1}{\sqrt2}\left  (1+\mathcal  O\left (  \vert \mu(h)\vert \right)+\mathcal O\big (e^{-\frac ch}\big ) \right)   \widetilde u_2\\
  \label{uh-fin2}
  &\quad +\, \mathcal  O\big (e^{-\frac ch}\big )  \ \text{  in } L^2_w(\Omega).
  \end{align} 
    From \eqref{uh-debut2},  one deduces the following proposition  which implies, using in addition  \eqref{nuh} and~\eqref{beta-h2-b}, the asymptotic estimates~\eqref{eq.QSD-1-h2} and \eqref{eq.QSD-2-h2}
    in Theorem~\ref{th.main2}.
   \begin{proposition} 
    \label{pr.proj'}
Let us assume that the  hypotheses  \textbf{[H-Well]} and \textbf{[H2]} are satisfied. Let  $u_h$ be the principal eigenfunction of $L_{f,h}^{D,(0)}$ satisfying \eqref{uh.norma} and let $(\widetilde u_{j})_{j\in\{1,2\}}$ be the functions introduced in~\eqref{eq.chi1}.  
Then, for any  open subset $\ft O$ of $\Omega$ and   for  $h>0$ small enough:
\begin{enumerate}
\item[i)] When, for some $i\in\{1,2\}$, $\ft O\cap \{x_{1},x_{2}\} =\{x_{i}\}$, it holds
\begin{align*}
\int  _{\ft O} u_h \ e^{- \frac{2}{h} f } &= \frac1{\sqrt2}\left (1+\mathcal   O( \vert \mu(h)\vert)+\mathcal O\big (e^{-\frac ch}\big )\right)\int  _{\ft O} \widetilde u_i \ e^{- \frac{2}{h} f } 
\\
&=\frac1{\sqrt2}\frac{(h \pi)^{\frac{d}{4} }}{\big(\det \Hess f(x_{i})\big)^{\frac14}}e^{-\frac{1}{h} \min_{\overline\Omega}f} 
\ \big(1+\mathcal O(\vert \mu(h)\vert)+\mathcal O(h)\big),
\end{align*}
where   $c>0$ is independent of $h$ and $ \mu(h)$ satisfies~\eqref{beta-h2-b}.  
\item[ii)] When $\ft O\cap \{x_{1},x_{2}\} =\emptyset$, one has
\begin{align*}
\int  _{\ft O} u_h \ e^{- \frac{2}{h} f } &= \mathcal O\left (e^{-\frac 1h( \min_{\overline\Omega}f+c)}\right ),  \, \text{  where $c>0$ is independent of $h$}.
\end{align*}

\end{enumerate}
\end{proposition}

\begin{proof}
The proof of Proposition \ref{pr.proj'} is similar to that one of Proposition~\ref{pr.proj}
using~\eqref{uh-fin2} instead of \eqref{uh-final-1}. 
\end{proof}

\begin{remark}
\label{re.sym2}
Let us  assume as in Remark~\ref{re.sym1} that there exists an isometry
$ \Phi:\overline\Omega\to\overline\Omega$ 
satisfying \eqref{eq.sym}
and denote by $\ft O_{1}\subset \Omega$ and $\ft O_{2}\subset \Omega$
 two disjoint  open sets  such that
$x_{i}\in \ft O_{i}$ for $i\in\{1,2\}$.
Using Proposition~\ref{pr.proj'} and the fact that, for every $h$ small enough,
$\widetilde u_{1}\circ \Phi=\widetilde u_{2}$, $\alpha_{1}(h)=\alpha_{2}(h)$ and hence $\mu(h)=0$,
it holds for $i\in\{1,2\}$:
\begin{align*}
\int  _{\Omega} u_h \ e^{- \frac{2}{h} f }&=
\big(1+\mathcal O(e^{-\frac ch})\big)\int  _{\ft O_{1}\cup \ft O_{2}} u_h \ e^{- \frac{2}{h} f }\\
&=\big(\sqrt 2+\mathcal O(e^{-\frac ch})\big)\int  _{ \ft O_{i}}  \widetilde u_{i} \ e^{- \frac{2}{h} f }.
\end{align*}
This implies the first part of Theorem~\ref{pr.main}. 
\end{remark}

\noindent
From \eqref{uh-debut2} and 
Proposition~\ref{interaction1}, one deduces 
 the following estimates.  

\begin{proposition} \label{lemm2'} 
Let us assume that the  hypotheses  \textbf{[H-Well]} and \textbf{[H2]} are satisfied.  Let  $u_h$ be the principal eigenfunction of $L_{f,h}^{D,(0)}$ satisfying \eqref{uh.norma}. Let moreover $(\widetilde u_{j})_{j\in\{1,2\}}$ and  $(\widetilde \psi_{i_{j}})_{(i,j)\in\bigcup_{p=1}^2 \{p\}\times \{1,\dots, \ft n_{p}\}}$ be as in Proposition~\ref{pr.Schuss}, and $(\psi_{i_{j}})_{(i,j)\in\bigcup_{p=1}^2 \{p\}\times \{1,\dots, \ft n_{p}\}}$ be  defined by~\eqref{eq.B0-B1}. Then, there exists $c>0$ such that  in the limit $h\to0$:
\begin{enumerate}
\item[i)] For every $k\in\{1,2\}$ and $j\in\{1,\dots,\ft n_{k}\}$,
  \begin{align*}
   \big \lp    \nabla   u_h,   \psi_{k_{j}}  \big \rp_{L^2_w}
&=-\frac{C_{k,j,k}}{\sqrt 2} \,h^{-\frac34}\,  e^{-\frac Hh }   \, \big (1+\mathcal O(\vert \mu(h)\vert) +   \mathcal   O(h )   \big),
\end{align*} 
where $C_{k,j,k} $ is defined in \eqref{eq.C} and $\mu(h)$  satisfies \eqref{beta-h2-b}. 

\item[ii)] For every  $j\in\{1,\dots,\ft m_{3}\}$, 
\begin{align*}
   \big \lp    \nabla   u_h,   \psi_{3_j}  \big \rp_{L^2_w}&=
\mathcal O\left(h^{-\frac12}\, e^{-\frac Hh } \right).
\end{align*} 
\item[iii)]  For every  $j\in\{\ft m_3+1,\dots,\ft n_{3}\}$,
  \begin{equation*}
   \big \lp    \nabla    u_h,   \psi_{3_j}  \big \rp_{L^2_w}   =\mathcal O\left( e^{-\frac{1}{h}(H+c)}  \right ).
  \end{equation*}  
  \end{enumerate}
   \end{proposition}
   
\begin{proof}[End of the proofs Theorems~\ref{th.main2} and \ref{pr.main}] Let us assume that the  hypotheses  \textbf{[H-Well]} and \textbf{[H2]} are satisfied. It remains to prove the asymptotic estimates~\eqref{eq.h2r} and \eqref{eq.h2r-b}.  We proceed as we did at the end of Section \ref{se-t1} to prove \eqref{t-1},~\eqref{t-2}, and~\eqref{t-3}.
Let us then consider   $F\in L^\infty(\pa \Omega,\mathbb R)$. 
\medskip

\noindent
Let us first assume that $\overline{\Sigma}$ does not contain any of the $z_{i,j}$'s for  $(i,j)\in\bigcup_{p=1}^2 \{p\}\times \{1,\dots, \ft n_{p}\}$. Then, using~\eqref{eq.dec-uh} together with Propositions~\ref{co.bdy-estim}~and~\ref{lemm2'}, one has in the limit $h\to 0$:
\begin{equation}
\label{eq.bdy1'}
\int_{\Sigma} F\,\partial_{n}u_h   e^{-\frac{2}{h} f} =\mathcal O \left(e^{-\frac 1h(\min_{\pa \Omega} f+H+c)} \right),
\end{equation}
\noindent
for some   $c>0$ independent of $h$.\\ Let us now assume that  $ {\Sigma}\cap\big \{z_{i,j}, (i,j)\in\bigcup_{p=1}^2 \{p\}\times \{1,\dots, \ft n_{p}\}\big\}=\{z_{p,\ell}\}$ and  that $F$ is $C^{\infty}$ around $z_{p,\ell}$. Then, using again~\eqref{eq.dec-uh} together with Propositions~\ref{co.bdy-estim}~and~\ref{lemm2'}, one has when $h\to 0$,
defining $H':=\min_{\pa \Omega} f+H$, 
\begin{align}
\nonumber
\int_{\Sigma} F\, \partial_{n}u_h\,   e^{-\frac{2}{h} f} &= \big\langle \nabla u_h , \psi_{p_\ell} \big\rp_{L^2_w} \ \int_{\Sigma}  F \,  \psi_{p_{\ell}} \cdot n \  e^{- \frac{2}{h}  f} +
\mathcal O\left( h^{\frac{d-6}{4}}   e^{-\frac 1h(H'+c)} \right )\\
\nonumber
&=
 - \sqrt2\, \partial_n f(z_{p,\ell})   \,  \frac{\big(\det \Hess f(x_{p})\big)^{\frac14}}{\big( \det \Hess f |_{ \partial \Omega}(z_{p,\ell})   \big)^{\frac12}} \, \pi ^{\frac{d-2}{4}}\, h^{\frac{d-6}{4}}\,  e^{-\frac {H'}h}\\
 \label{eq.bdy2'}
 &\quad   \times\big (F(z_{p,\ell})+\mathcal O(\vert \mu(h)\vert) +   \mathcal   O(h )   \big),
\end{align}
where $c>0$ is independent of $h$ and~$\mu(h)$  satisfies~\eqref{beta-h2-b}. The asymptotic estimates \eqref{eq.h2r} and \eqref{eq.h2r-b} are then straightforward consequences of  
\eqref{eq.loi-bord} and
\eqref{eq.bdy1'}, \eqref{eq.bdy2'}, using in addition
\eqref{lh1-h2-c1},~\eqref{lh1-h2-c2},  and Proposition~\ref{pr.proj'}.  This concludes the proof of Theorems~\ref{th.main2} and \ref{pr.main}.
   \end{proof}

\bibliography{double-well} 
\bibliographystyle{plain} 

\end{document}